\documentclass[a4paper,10pt, oneside]{article}
\usepackage{amsmath, amssymb, amsthm,graphicx}
\usepackage[font=small,labelfont=md,textfont=it]{caption}
\usepackage{floatrow,float}
\usepackage[titletoc, title]{appendix}
\usepackage[colorlinks,linkcolor=blue,citecolor=blue]{hyperref}
\usepackage{longtable}
\usepackage{pgfplots}
\usepackage{diagbox}
\usepackage{booktabs,makecell,multirow}
\usepackage[capitalise, nosort]{cleveref}
\usepackage{cases,color}
\crefname{equation}{}{}
\crefname{lem}{Lemma}{Lemmas}
\crefname{thm}{Theorem}{Theorems}

\DeclareMathOperator{\D}{D}

\newcommand{\dual}[1]{\left\langle {#1} \right\rangle}

\newcommand{\nm}[1]{\left\lVert {#1} \right\rVert}
\newcommand{\snm}[1]{\left\lvert {#1} \right\rvert}
\newcommand{\ssnm}[1]
{
  \left\vert\kern-0.25ex
  \left\vert\kern-0.25ex
  \left\vert
  {#1}
  \right\vert\kern-0.25ex
  \right\vert\kern-0.25ex
  \right\vert
}

\makeatletter
\def\spher@harm#1{%
  \vbox{\hbox{%
    \offinterlineskip
    \valign{&\hb@xt@2\p@{\hss$##$\hss}\vskip.2ex\cr#1\crcr}%
  }\vskip-.36ex}%
}
\def\gshone{\spher@harm{.}}
\def\gshtwo{\spher@harm{.&.}}
\def\gshthree{\spher@harm{.&.&.}}
\let\gsh\spher@harm
\makeatother

\newtheorem{Def}{Definition}[section]

\newtheorem{lem}{Lemma}[section]
\newtheorem{rem}{Remark}[section]
\newtheorem{thm}{Theorem}[section]

\makeatletter\def\@captype{table}\makeatother

\begin{document}

\title{
  \Large \bf Convergence analysis of a Petrov-Galerkin
  method for fractional wave problems with nonsmooth data
  \thanks
  {
    This work was supported in part
    by National Natural Science Foundation
    of China (11771312).
  }
}
\author{
  Hao Luo \thanks{Email: galeolev@foxmail.com},
  Binjie Li \thanks{Corresponding author. Email: libinjie@scu.edu.cn},
  Xiaoping Xie \thanks{Email: xpxie@scu.edu.cn} \\
  {School of Mathematics, Sichuan University, Chengdu 610064, China}
}

\date{}
\maketitle

\begin{abstract}
  This paper analyzes the convergence of a Petrov-Galerkin method for time fractional
  wave problems with nonsmooth data. Well-posedness and regularity of the weak
  solution to the time fractional wave problem are firstly established. Then an
  optimal convergence analysis with nonsmooth data is derived. Moreover, several
  numerical experiments are presented to validate the theoretical results.
\end{abstract}

\medskip\noindent{\bf Keywords:}
fractional wave problem, regularity,
Petrov-Galerkin,
convergence analysis,
nonsmooth data.

\section{Introduction}
Let $ T>0 $ be a given time and $ \Omega \subset \mathbb R^d $ ($d=1,2,3$) be a
convex $d$-polytope. This paper considers the following time fractional wave problem:
\begin{equation}
  \label{eq:model}
  \left\{
    \begin{aligned}
      \D_{0+}^\alpha (u-u_0-tu_1)- \Delta u & = f  &  & \text{in $~~ \Omega\times(0,T)$,}         \\
      u                                 & = 0   &  & \text{on $ \partial\Omega \times (0,T) $,} \\
      u(0)                              & = u_0&  & \text{in $~~ \Omega $,}\\
      u_t(0)                             & = u_1 &  & \text{in $~~ \Omega $,}
    \end{aligned}
  \right.
\end{equation}
where $ 1<\alpha<2 $, $ \D_{0+}^\alpha $ is a Riemann-Liouville fractional differential operator of order $ \alpha $, and $ u_0,\,u_1$ and $f $ are given data.

In recent years, the time fractional wave problem \cref{eq:model} has attracted much
attention. It has been applied to model the anomalous process which may occur in
anomalous transport or diffusion in heterogeneous media \cite{Luchko2011}. In
addition, the solution to the time fractional wave problem governs the propagation of
stress waves in viscoelastic media \cite{Mainardi1996,Mainardi2001}. For more details
related to the applications of problem \cref{eq:model}, we refer the reader to
\cite{Kilbas2006,Podlubny1998}.

Let us first summarize some regularity results of the fractional wave problem. In
\cite{Bazhlekova1996}, Bazhlekova considered the Duhamel-type representation of the
solution to the fractional wave equation by the Mittag-Leffler function; however, the
author did not investigate the regularity of the solution. Later on, in \cite{Bazhlekova2001}, Bazhlekova obtained the maximal $ L^p $-regularity estimate
\[\small
\begin{split}
  \nm{u}_{L^p(0,T;L^q(\Omega))} +
\nm{\D_{0+}^\alpha u}_{L^p(0,T;L^q(\Omega))} +
\nm{\Delta u}_{L^p(0,T;L^q(\Omega))} \leqslant
C \nm{f}_{L^p(0,T;L^q(\Omega))},
\end{split}
\]
where $ 1 < p,q < \infty $. Sakamoto et al.~\cite{Sakamoto2011} introduced a weak solution to the fractional wave equation by means of the eigenfunction expansions. They
established the well-posedness of the weak solution and derived several regularity estimates in the continuous vector-valued spaces.

Then, let us review the numerical treatments for the fractional wave equation. In
\cite{Oldham1974}, two kinds of finite difference methods for the computation of
fractional derivatives were presented: the first method, called $ L $-type scheme,
uses the Lagrange interpolation technique; the second one, called $ G $-type method,
is based on the Gr\"{u}nwald-Letnikov definition. Sun et al.~\cite{Sun2006} developed
a Crank-Nicolson scheme by the $L1$-scheme for the fractional wave equation and
derived the convergence order $\mathcal O(\tau^{3-\alpha})$ for $C^3$ solutions. Jin
et al.~\cite{Jin2016} analyzed the $G1$-method and the second-order backward
difference method for fractional wave equations, and they obtained the accuracies
$\mathcal O(\tau)$ and $\mathcal O(\tau^2)$, respectively. In our previous work
\cite{Li2018}, a time-spectral method for fractional wave problems was designed,
which possesses exponential decay in temporal discretization, under the condition
that the solution is smooth enough. Recently, to conquer the singularity in time
variable, Li et al.~\cite{Li2017As} presented a space-time finite element method for
problem \cref{eq:model}, and proved that high-order temporal accuracy can still be
achieved if appropriate graded temporal grids are adopted. Under some conditions, problem \cref{eq:model} is equivalent to an integro-differential model, and there are many works on the numerical methods for this model; see \cite{Lubich1996,Cuesta2006,Mustapha2012Superconvergence} and the references therein. To our knowledge, except
for \cite{Jin2016}, no work available is devoted to the numerical analysis for
problem \cref{eq:model} with nonsmooth data.

This motivates us to consider the numerical analysis for problem \cref{eq:model} with
low regularity data. In this paper,
we first introduce a weak solution of problem \cref{eq:model}
by the variational approach and establish the regularity
results of the weak solution in the case $ u_0 = u_1 = 0 $.
Then by means of the famous transposition method \cite{Lions-I},
the weak solution and its regularity of problem
\cref{eq:model} are also considered with more general data.
Finally, under the condition that $ u_0=u_1=0 $,
for a Petrov-Galerkin method we obtain the following
error estimates:
\begin{itemize}
    \item if $ f \in L^{2}(0,T;L^2(\Omega)) $, then
    \begin{equation}
    \label{eq:est1}
    \begin{aligned}
    {}& \nm{(u-U)'}_{{}_0H^{(\alpha-1)/2}(0,T;L^2(\Omega))} +
    \nm{u-U}_{ C([0,T];\dot H^1(\Omega))}  \\
    \leqslant{}&C \big( \tau^{(\alpha-1)/2} +   \eta_1(\alpha,\tau,h) \big)
    \nm{f}_{L^{2}(0,T;L^2(\Omega))},
    \end{aligned}
    \end{equation}
    where
    \begin{equation*}
    \eta_1(\alpha,\tau,h):=
    \left\{
    \begin{aligned}
    &h^{1-1/\alpha}&&{\text {if}~} 1<\alpha\leqslant 3/2,\\
    &\tau^{-1/2}h&&{\text {if}~} 3/2<\alpha<2;
    \end{aligned}
    \right.
    \end{equation*}
    \item if $ f \in {}_0H^{2-\alpha}(0,T;L^2(\Omega)) $, then
    \begin{equation}    \label{eq:est2}
    \small
    \begin{split}
        \nm{u-U}_{ C([0,T];\dot H^1(\Omega))}
    \leqslant C
    \big(\tau^{(3-\alpha)/2} +\eta_2(\alpha,\tau,h)\big)
    \nm{f}_{{}_0\!H^{2-\alpha}(0,T;L^2(\Omega))},
    \end{split}
    \end{equation}
        \begin{equation}    \label{eq:est3}
    \small
    \begin{split}
\nm{u'\!\!-\!U'}_{{}_0\!H^{(\alpha-1)/2}(0,T;L^2(\Omega))}
\!\leqslant \!
C \big( \tau^{(3-\alpha)/2} \!\!+\! \eta_3(\alpha,\tau,h)\big)\!
\nm{f}_{{}_0\!H^{2
        \!-
        \!\alpha}(0,T;L^2(\Omega))},
    \end{split}
    \end{equation}
    where
    \begin{equation*}
    \eta_2(\alpha,\tau,h):=
    \left\{
    \begin{aligned}
    &h&&{\text {if}~} 1<\alpha< 3/2,\\
    &  \big(1+\snm{\log h} \big)h&&{\text {if}~} \alpha= 3/2,\\
    &h^{3/\alpha-1}+\tau^{3/2-\alpha}h  &&{\text {if}~}3/2<\alpha<2,
    \end{aligned}
    \right.
    \end{equation*}
    and
    \begin{equation*}
    \eta_3(\alpha,\tau,h):=
    \left\{
    \begin{aligned}
    &h^{3/\alpha-1}&&{\text {if}~}1<\alpha\leqslant 3/2,\\
    &  \tau^{3/2-\alpha}h&&{\text {if}~}3/2<\alpha<2.
    \end{aligned}
    \right.
    \end{equation*}
\end{itemize}
We note that, if $1<\alpha\leqslant3/2$ then estimates \cref{eq:est1,eq:est3}
are optimal with respect to the regularity of $u$ and  \cref{eq:est2} is optimal and nearly optimal with respect to the regularity of $ u $ for $ 1<\alpha <3/2 $ and $ \alpha = 3/2 $, respectively. This is verified by our numerical experiments. If $3/2<\alpha<2$, then all the estimates \cref{eq:est1,eq:est2,eq:est3} are optimal with respect to the regularity of $u$ provided that $h\leqslant C\tau^{\alpha/2}$. However, numerical results also indicate the optimal accuracy with respect to the regularity without this requirement.

The remainder of this paper consists of five sections. Firstly, some conventions and
Sobolev spaces are introduced in \cref{sec:pre}. Secondly, several fundamental
properties of the fractional calculus operators are summarized in
\cref{sec:frac-cal}. Thirdly, the well-posedness and regularity of the weak solution
to problem \cref{eq:model} are rigorously established in \cref{sec:regu}. Fourthly,
the convergence of a Petrov-Galerkin method is derived in \cref{sec:fem}.
Finally, in \cref{sec:numer}
numerical experiments are presented to verify the theoretical results and \cref{sec:con} provides some concluding remarks.

\section{Preliminary}
\label{sec:pre}
First of all, let us introduce some conventions: for a Lebesgue measurable set $
\omega $ of $ \mathbb R^l $ ($ l= 1,2,3,4 $), $ H^\gamma(\omega) $ ($
\gamma \in\mathbb R$) and $ H_0^\beta(\omega) $ ($ \beta>0$) denote two
standard Sobolev spaces \cite[Chapter 34]{Tartar2007} and the symbol $ \dual{p,q}_\omega $
means $ \int_\omega pq $ whenever $pq\in L^1(\omega)$; for a Banach space $ X $, $ X^* $ is the dual space of $ X
$ and $ \dual{\cdot,\cdot}_X $ means the duality pairing between $ X^* $ and $ X $;
if $ X $ and $ Y $ are two Banach spaces, then $ [X,Y]_{\theta,2} $ is the
interpolation space constructed by the famous $ K $-method \cite[Chapter
22]{Tartar2007};
the symbol $C_\times$ denotes a generic positive constant depending
only on its subscript(s) $\times$, and its value may differ at each occurrence.

Next, we form some Hilbert spaces on the eigenvectors of $-\Delta$ and present some
basic properties of these spaces. It is well known that there exists an orthonormal
basis $\{\phi_n: n \in \mathbb N \}$ of $ L^2(\Omega) $ such that
\[
  \left\{
    \begin{aligned}
      -\Delta \phi_n ={} &\lambda_n \phi_n&&\,
      {\rm~in~}~\,\,\Omega,\\
      \phi_n={}&0&&{\rm~on~}\partial\Omega,
    \end{aligned}
  \right.
\]
where $ \{ \lambda_n: n \in \mathbb N \} $ is a positive non-decreasing sequence and
$\lambda_n\to\infty$ as $n\to\infty$. For any $ \gamma \in\mathbb R $, define
\[
  \dot H^\gamma(\Omega) := \left\{
    \sum_{n=0}^\infty c_n \phi_n:\
    \sum_{n=0}^\infty \lambda_n^\gamma c_n^2 < \infty
  \right\}
\]
and endow this space with the inner product
\[
\left(\sum_{n=0}^\infty c_n \phi_n,
\sum_{n=0}^\infty d_n \phi_n\right)_{
    \dot H^\gamma(\Omega) }:=
\sum_{n=0}^\infty \lambda_n^\gamma c_nd_n,
\]
for all $\sum_{n=0}^\infty c_n \phi_n,\,\sum_{n=0}^\infty d_n \phi_n\in\dot H^\gamma(\Omega)$. Denote by $\nm{\cdot}_{\dot H^\gamma(\Omega) }$ the induced norm with respect to this inner product. We see that $\dot H^\gamma(\Omega)$ is a separable Hilbert space with an orthonormal basis $\{\lambda_n^{-\gamma/2}\phi_n:n\in\mathbb N\}$ and the space $\dot H^{-\gamma}(\Omega)$ is the dual space of $\dot H^{\gamma}(\Omega)$ in the following sense
\[
\dual{\sum_{n=0}^\infty c_n \phi_n,
\sum_{n=0}^\infty d_n \phi_n}_{
    \dot H^\gamma(\Omega) }:=
\sum_{n=0}^\infty c_nd_n,
\]
for all $\sum_{n=0}^\infty c_n \phi_n\in\dot H^{-\gamma}(\Omega)$ and $ \sum_{n=0}^\infty d_n \phi_n\in\dot H^\gamma(\Omega)$.
Furthermore, it is clear that $ \dot H^0(\Omega)=L^2(\Omega) $ and $ \dot H^1(\Omega) $ coincides
with $ H_0^1(\Omega) $ with equivalent norms. Hence, for $ 0 < \gamma < 1 $,
by the theory of interpolation
spaces \cite{Tartar2007}, $ \dot H^\gamma(\Omega) $ coincides
with $ H_0^\gamma(\Omega) = [L^2(\Omega),H_0^1(\Omega)]_{\gamma,2} $ with equivalent
norms. As \cite[Corollary 9.1.23]{Hackbusch1992} implies
\[
  \nm{v}_{H^2(\Omega)} \leqslant C_\Omega
  \nm{v}_{\dot H^2(\Omega)}
  \quad \forall\, v \in \dot H^2(\Omega),
\]
the space $ \dot H^\gamma(\Omega) $ is continuously embedded
into $ H_0^1(\Omega) \cap H^\gamma(\Omega) $ if $ 1 < \gamma < 2 $.

In the rest of this section, assume that $-\infty<a<b<\infty$. Now we introduce some Sobolev spaces as follows. For any $ m \in\mathbb
N $, define
\begin{align*}
{}^0\!H^m(a,b):={}&\{v\in H^m(a,b): v^{(k)}(b)=0,\,\, 0\leqslant k<m,~k\in\mathbb N\},\\
{}_0H^m(a,b):={}&\{v\in H^m(a,b): v^{(k)}(a)=0,\,\, 0\leqslant k<m,~k\in\mathbb N\},
\end{align*}
where $ v^{(k)} $ is the $ k $-th weak derivative of $ v $, and endow those two spaces with the following norms
\begin{align*}
\nm{v}_{{}^0\!H^m(a,b)}:={}&\big\|v^{(m)}\big\|_{L^2(a,b)},\quad\forall\,v\in{}^0\!H^m(a,b),\\
\nm{v}_{{}_0H^m(a,b)}:={}&\big\|v^{(m)}\big\|_{L^2(a,b)},\quad\forall\,v\in{}_0H^m(a,b),
\end{align*}
respectively. For $k-1<\gamma<k,~k\in\mathbb N_{>0}$, define
\begin{align*}
{}^0\!H^\gamma(a,b) &:= [ {}^0\!H^{k-1}(a,b) , {}^0\!H^k(a,b) ]_{\gamma -k+1 ,2},\\
{}_0H^\gamma(a,b) &:= [ {}_0H^{k-1}(a,b), {}_0H^k(a,b) ]_{\gamma -k+1 ,2}.
\end{align*}
By \cite[Chapter 1]{Lions-I}, we have the following standard
results: if $ 0<\gamma<1/2 $, then
$ {}^0\!H^\gamma(a,b) $, $ {}_0H^\gamma(a,b) $ and $
H^\gamma(a,b) $ are equivalent;
if $m+1/2<\gamma< m+1,m\in\mathbb N$, then
\begin{align*}
{}^0\!H^\gamma(a,b)={}&\{v\in {}^0\!H^m(a,b):v^{(m)}(b) =0,\, (b-t)^{m-\gamma}v^{(m)}\in L^{2}(a,b)\},\\
{}_0H^\gamma(a,b)={}&\{v\in {}_0H^m(a,b): v^{(m)}(a) =0,\,(t-a)^{m-\gamma}
v^{(m)}\in L^{2}(a,b)\},
\end{align*}
with equivalent norms;
if $m\leqslant\gamma\leqslant m+1/2,m\in\mathbb N$, then
\begin{align*}
{}^0\!H^\gamma(a,b)={}&\{v\in {}^0\!H^m(a,b): (b-t)^{m-\gamma}v^{(m)}\in L^{2}(a,b)\},\\
{}_0H^\gamma(a,b)={}&\{v\in {}_0H^m(a,b): (t-a)^{m-\gamma}
v^{(m)}\in L^{2}(a,b)\},
\end{align*}
in the sense of equivalent norms.
For $\gamma\geqslant0$, the spaces ${}^0\!H^\gamma(a,b)$ and ${}_0H^\gamma(a,b) $ can be defined
equivalently as the domains of fractional power of second order differential operators (see \cite{Yamamoto2018,Gorenflo1999Operator,Gorenflo2015}).
For $\gamma>0$, denote by ${}_0H^{-\gamma}(a,b)$ and ${}^0\!H^{-\gamma}(a,b)$ the dual spaces of ${}^0\!H^{\gamma}(a,b)$ and ${}_0H^{\gamma}(a,b)$, respectively. Since ${}^0\!H^{\gamma}(a,b)$ and ${}_0H^{\gamma}(a,b)$ are reflexive, they are the dual spaces of ${}_0H^{-\gamma}(a,b)$ and ${}^0\!H^{-\gamma}(a,b)$, respectively. Moreover, by \cite[Theorems~1.18 and 1.23]{Lunardi2018} and Theorems 12.2-12.6 of
\cite[Chapter~1]{Lions-I}, we readily conclude the following lemma.
\begin{lem}
    \label{lem:inter_space}
    If $ 0 < \theta < 1 $ and $ \beta,\gamma\in\mathbb R $ then
    \begin{align}
        \big[ {}^0\!H^\beta(a,b), {}^0\!H^\gamma(a,b) \big]_{\theta,2} ={}&
    {}^0\!H^{(1-\theta)\beta + \theta \gamma}(a,b),\notag\\
        \big[ {}_0H^\beta(a,b), {}_0H^\gamma(a,b) \big]_{\theta,2} ={}&
    {}_0H^{(1-\theta)\beta + \theta \gamma}(a,b),
    \label{eq:inter_space-1}
    \end{align}
    with equivalent norms.
\end{lem}
\begin{rem}
    We will give more details about how to derive \cref{eq:inter_space-1}. As $
    (H^1(a,b))^* $ and $ H_0^1(a,b) $ are continuously embedded in $ {}_0H^{-1}(a,b) $
    and $ {}_0H^1(a,b) $, respectively, by Theorem 12.3 of \cite[Chapter~1]{Lions-I}
    we have that $ L^2(a,b) $ is continuously embedded in $ [{}_0H^{-1}(a,b),
    {}_0H^1(a,b)]_{1/2,2} $. Conversely, since $ {}_0H^{-1}(a,b) $ and $ {}_0H^1(a,b) $
    are continuously embedded in $ H^{-1}(a,b) $ and $ H^1(a,b) $, respectively, by
    Theorem 12.4 of \cite[Chapter~1]{Lions-I} we have that $ [{}_0H^{-1}(a,b),
    {}_0H^1(a,b)]_{1/2,2} $ is continuously embedded in $ L^2(a,b) $. Therefore, $
    [{}_0H^{-1}(a,b),{}_0H^1(a,b)]_{1/2,2} = L^2(a,b) $ with equivalent norms. Then by
    \cite[Theorem 1.23]{Lunardi2018} we obtain that, for any $ 0 < \theta < 1 $,
    \[
    \begin{split}\small
    {}_0H^\theta(a,b)  & = [L^2(a,b), {}_0H^1(a,b)]_{\theta,2}  =
    \big[ [{}_0H^{-1}(a,b), {}_0H^1(a,b)]_{1/2,2}, {}_0H^1(a,b) \big]_{\theta,2} \\
    & = [{}_0H^{-1}(a,b), {}_0H^1(a,b)]_{(1+\theta)/2,2},
    \end{split}
    \]
    with equivalent norms. The other cases are derived similarly.
\end{rem}

Finally, let us introduce some vector-valued spaces.
Let $X$ be a separable Hilbert space with an orthonormal basis $\{e_n:n\in\mathbb N\}$. For any $ \gamma\in\mathbb R $, define
\[
{}_0H^{\gamma}(a,b;X):=\left\{\sum_{n=0}^\infty c_ne_n:\sum_{n=0}^\infty\nm{c_n}_{{}_0H^\gamma(a,b)}^2<\infty\right\},
\]
and endow this space with the norm
\[
\bigg\|\sum_{n=0}^\infty c_ne_n\bigg\|_{{}_0H^{\gamma}(a,b;X)}:=\left(\sum_{n=0}^\infty\nm{c_n}_{{}_0H^\gamma(a,b)}^2\right)^{1/2}.
\]
The space $ {}^0\!H^\gamma(a,b;X) $ can be defined analogously. It is evident that both ${}_0H^{\gamma}(a,b;X)$ and $ {}^0\!H^\gamma(a,b;X) $ are reflexive. In addition, the space $ {}^0\!H^{-\gamma}(a,b;X) $ is the dual space of $ {}_0H^\gamma(a,b;X) $ in the sense that
\[
\dual{\sum_{n=0}^\infty c_ne_n,\sum_{n=0}^\infty d_ne_n}_{ {}_0H^\gamma(a,b;X)} :=
\sum_{n=0}^\infty \dual{c_n,d_n}_{{}_0H^\gamma(a,b)},
\]
for all $\sum_{n=0}^\infty c_ne_n\in{}^0\!H^{-\gamma}(a,b;X) $ and $\sum_{n=0}^\infty d_ne_n\in{}_0H^\gamma(a,b;X) $, and
the space $ {}_0H^{-\gamma}(a,b;X) $ is the dual space of $ {}^0\!H^\gamma(a,b;X) $ in the sense that
\[
\dual{\sum_{n=0}^\infty c_ne_n,
    \sum_{n=0}^\infty d_ne_n}_{ {}^0\!H^\gamma(a,b;X)} :=
\sum_{n=0}^\infty \dual{c_n,d_n}_{{}^0\!H^\gamma(a,b)},
\]
for all $\sum_{n=0}^\infty c_ne_n\in{}_0H^{-\gamma}(a,b;X) $ and $\sum_{n=0}^\infty d_ne_n\in{}^0\!H^\gamma(a,b;X) $.
Moreover, we use $C([a,b];X)$ to denote the continuous $X$-valued space.
\begin{lem}
    \label{lem:interp}
    Assume that $ s,r,\beta ,\gamma\in\mathbb R $ and $0<\theta<1$.
    If $ v \in
    {}_0H^\beta(0,1;\dot H^r(\Omega))
    \cap {}_0H^\gamma(0,1;\dot H^s(\Omega)) $, then
    \begin{equation}
    \label{eq:interp-1}
    \begin{split}
    {}& \nm{v}_{
        {}_0H^{(1-\theta)\beta + \theta \gamma}
        (0,1;\dot H^{(1-\theta)r+\theta s}(\Omega))} \\
    \leqslant{} &   C_{\beta,\gamma,\theta}
    \nm{v}_{{}_0H^\beta(0,1;\dot H^r(\Omega)))}^{1-\theta}
    \nm{v}_{{}_0H^\gamma(0,1;\dot H^s(\Omega))}^\theta.
    \end{split}
    \end{equation}
\end{lem}
\begin{proof}
    By definition, there exists a unique decomposition $ v = \sum_{n = 0}^\infty v_n \phi_n $, such that
    \begin{align*}
    \nm{v}_{{}_0H^\beta(0,1;\dot H^r(\Omega))}^2 =
    \sum_{n=0}^\infty \lambda_n^r \nm{v_n}_{{}_0H^\beta(0,1)}^2, \\
    \nm{v}_{{}_0H^\gamma(0,1;\dot H^s(\Omega))}^2 =
    \sum_{n=0}^\infty \lambda_n^s \nm{v_n}_{{}_0H^\gamma(0,1)}^2.
    \end{align*}
    Therefore, by \cite[Corollary 1.7]{Lunardi2018} and \cref{lem:inter_space} we have
    \begin{align*}
    & \nm{v}_{
        {}_0H^{(1-\theta)\beta + \theta \gamma}
        (0,1;\dot H^{(1-\theta)r + \theta s}(\Omega))
    }^2 \\
    ={} &
    \sum_{n=0}^\infty \lambda_n^{(1-\theta)r + \theta s}
    \nm{v_n}_{{}_0H^{(1-\theta)\beta + \theta \gamma}(0,1)}^2 \\
    \leqslant{} & C_{\beta,\gamma,\theta}
    \sum_{n=0}^\infty \lambda_n^{(1-\theta)r + \theta s}
    \nm{v_n}_{[{}_0H^{\beta}(0,1),{}_0H^{\gamma}(0,1)]_{\theta,2}}^2 \\
    \leqslant{} &
C_{\beta,\gamma,\theta}
    \sum_{n=0}^\infty \lambda_n^{(1-\theta)r + \theta s}
    \nm{v_n}_{{}_0H^\beta(0,1)}^{2(1-\theta)}
    \nm{v_n}_{{}_0H^\gamma(0,1)}^{2\theta} \\
    ={} &
    C_{\beta,\gamma,\theta}
    \sum_{n=0}^\infty
    \big( \lambda_n^r\nm{v_n}_{{}_0H^\beta(0,1)}^2 \big)^{1-\theta}
    \big( \lambda_n^s\nm{v_n}_{{}_0H^\gamma(0,1)}^2 \big)^\theta \\
    \leqslant{} & C_{\beta,\gamma,\theta}
    \nm{v}_{{}_0H^\beta(0,1;\dot H^r(\Omega))}^{2(1-\theta)}
    \nm{v}_{{}_0H^\gamma(0,1;\dot H^s(\Omega))}^{2\theta},
    \end{align*}
    which implies \cref{eq:interp-1}.
\end{proof}
\section{Fractional Calculus Operators}
\label{sec:frac-cal}
In this section, we firstly summarize several fundamental properties of fractional calculus
operators, then we generalize the fractional integral operator and prove some useful results. Assume that $ -\infty<a<b<\infty $ and $ X $ is a separable Hilbert space.
\begin{Def}
  \label{def:frac_calc}
  For $ \gamma >0$, define
  \begin{align*}
    \left(\D_{a+}^{-\gamma} v\right)(t) &:=
    \frac1{ \Gamma(\gamma) }
    \int_a^t (t-s)^{\gamma-1} v(s) \, \mathrm{d}s, \quad t\in(a,b), \\
    \left(\D_{b-}^{-\gamma} v\right)(t) &:=
    \frac1{ \Gamma(\gamma) }
    \int_t^b (s-t)^{\gamma-1} v(s) \, \mathrm{d}s, \quad t\in(a,b),
  \end{align*}
  for all $ v \in L^1(a,b;X) $, where $ \Gamma(\cdot) $ is the Gamma function and $ L^1(a,b;X) $ denotes the X-valued Bochner integrable space. In
  addition, let $ \D_{a+}^0 $ and $ \D_{b-}^0 $
  be the identity operator on $
  L^1(a,b;X) $. For $ j-1 < \gamma \leqslant j $ with $ j \in \mathbb N_{>0}  $,
  define
  \begin{align*}
    \D_{a+}^\gamma v & := \D^j \D_{a+}^{\gamma-j}v, \\
    \D_{b-}^\gamma v & := (-\D)^j \D_{b-}^{\gamma-j}v,
  \end{align*}
  for all $ v \in L^1(a,b;X) $, where $ \D $ is the first-order differential operator
  in the distribution sense.
\end{Def}
\begin{lem}[\cite{Samko1993}]
    \label{lem:basic-frac}
    Let $v\in L^1(a,b)$.
    If $  \gamma,\beta \geqslant0 $, then
    \[
    \D_{a+}^{-\gamma} \D_{a+}^{-\beta}v = \D_{a+}^{-\gamma-\beta}v, \quad
    \D_{b-}^{-\gamma} \D_{b-}^{-\beta}v= \D_{b-}^{-\gamma-\beta}v.
    \]
    If $  \gamma\geqslant \beta >0 $, then
\[
\D_{a+}^\gamma \D_{a+}^{-\beta}v = \D_{a+}^{\gamma-\beta}v, \quad
\D_{b-}^\gamma \D_{b-}^{-\beta}v = \D_{b-}^{\gamma-\beta}v.
\]
\end{lem}

\begin{lem}[\cite{Kilbas2006}]
    \label{lem:basic-2}
    Assume that $ \gamma \geqslant0$.
    If $ w, v \in L^2(a,b) $, then
    \[
    \big\langle \D_{a+}^{-\gamma} w, v   \big\rangle_{(a,b)} =
    \big\langle w, \D_{b-}^{-\gamma} v   \big\rangle_{(a,b)}.
    \]
    If $ v \in L^2(a,b) $, then
    \begin{align*}
    \big\|\D_{a+}^{-\gamma} v    \big\|_{L^2(a,b)}
    &\leqslant \frac{(b-a)^{\gamma}}{\Gamma(\gamma+1)}
    \nm{v}_{L^2(a,b)}, \\
    \big\|\D_{b-}^{-\gamma} v    \big\|_{L^2(a,b)}
    &\leqslant  \frac{(b-a)^{\gamma}}{\Gamma(\gamma+1)}
    \nm{v}_{L^2(a,b)}.
    \end{align*}
\end{lem}
\begin{lem}
    \label{lem:regu-frac-int-leq}
    If $ \gamma,\beta\geqslant0  $, then
    \begin{align}
        \nm{\D_{a+}^{-\gamma} v}_{{}_0H^{\beta+\gamma}(a,b)}  \leqslant{}&
    C_{\beta,\gamma} \nm{v}_{{}_0H^\beta(a,b)} \quad
    \forall\,v \in {}_0H^\beta(a,b),
    \label{eq:regu-left}\\
        \nm{\D_{b-}^{-\gamma} v}_{{}^0\!H^{\beta+\gamma}(a,b)}  \leqslant{}&
    C_{\beta,\gamma}\nm{v}_{{}^0\!H^\beta(a,b)}\quad
    \forall\,v \in {}^0\!H^\beta(a,b).
    \label{eq:regu-right}
    \end{align}
\end{lem}
\begin{proof}
    As the proof of \cref{eq:regu-right} is analogous to that of \cref{eq:regu-left} and the case $\gamma,\beta\in\mathbb N$ is trivial, we only prove \cref{eq:regu-left} for the case that $\gamma\notin\mathbb N$ or $\beta\notin\mathbb N$.

    We first use the standard scaling argument to prove the case $\beta=0$ and $0<\gamma<1$. By definition we have
    \[
    {}_0H^{\gamma}(a,b)=\big[ {}_0H^{0}(a,b),~    {}_0H^{1}(a,b)\big]_{\gamma,2},
    \]
    and this space is endowed with the following norm
    \[
    \nm{w}_{{}_0H^{\gamma}(a,b)} = \left(
    \int_0^\infty \big(t^{-\gamma}K(t,w)\big)^2 \, \frac{\mathrm{d}t}t
    \right)^{1/2} \quad \forall\,w\in {}_0H^{\gamma}(a,b),
    \]
    where
    \[
    K(t,w) := \inf_{
        \substack{
            w = w_0 + w_1 \\
            w_0 \in {}_0\!H^{0}(a,b),\,
            w_1 \in {}_0\!H^{1}(a,b)
        }
    }
    \nm{w_0}_{{}_0\!H^{0}(a,b)} + t \nm{w_1}_{{}_0\!H^{1}(a,b)},
    \quad 0<t<\infty,
    \]
    for all $w\in {}_0H^{\gamma}(a,b)$. For any $v\in{}_0H^0(a,b)$, define
    \[
    \widehat{v}(s) := v\big(a+(b-a)s\big),\quad 0<s<1,
    \]
    then a direct computation gives
    \[
    K(t,\D_{a+}^{-\gamma} v) = (b-a)^{1/2+\gamma}K\big(t/(b-a),\D_{0+}^{-\gamma}\widehat{v}\big),\quad 0<t<\infty.
    \]
    Since using \cite[Lemma A.4]{Li2017As} gives
    \[
    \nm{\D_{0+}^{-\gamma}w}_{{}_0H^{\gamma}(0,1)}
    \leqslant C_{\gamma}
    \nm{w}_{{}_0H^{0}(0,1)}\quad\forall\, w\in{}_0H^{0}(0,1),
    \]
    it follows that
    \[
    \begin{split}
    {}&\nm{\D_{a+}^{-\gamma} v}_{{}_0H^{\gamma}(a,b)} = (b-a)^{1/2}\nm{\D_{0+}^{-\gamma}
        \widehat{v}}_{{}_0H^{\gamma}(0,1)}\\
    \leqslant {}&C_{\gamma}(b-a)^{1/2}\nm{\widehat{v}}_{{}_0H^{0}(0,1)}
    ={}C_{\gamma}\nm{v}_{{}_0H^{0}(a,b)}.
    \end{split}
    \]
     This proves \cref{eq:regu-left} for $\beta=0$ and $0<\gamma<1$.

    Then we consider the case $\beta\in\mathbb N$ and $m<\gamma<m+1,m\in\mathbb N$. Since $v\in{}_0 H^{\beta}(a,b)$, by \cref{lem:basic-frac}, it is evident that
    \[
\D_{a+}^{m+\beta}\D_{a+}^{-\gamma} v = \D_{a+}^{m+\beta-\gamma}\D_{a+}^{-\beta} \D_{a+}^{\beta}v=
\D_{a+}^{m-\gamma} \D_{a+}^{\beta}v.
    \]
    Therefore, a direct calculation yields
    \[
    \begin{split}
        K(t,\D_{a+}^{-\gamma} v) ={}& \inf_{
        \substack{
            \D_{a+}^{-\gamma} v = w_0 + w_1 \\
            w_0 \in {}_0\!H^{m+\beta}(a,b),\\
            w_1 \in {}_0\!H^{m+\beta+1}(a,b)
        }
    }
    \nm{w_0}_{{}_0\!H^{m+\beta}(a,b)} + t \nm{w_1}_{{}_0\!H^{m+\beta+1}(a,b)}\\
={}& \inf_{
        \substack{
            \D_{a+}^{m-\gamma} \D_{a+}^{\beta}v= v_0 + v_1 \\
            v_0 \in {}_0\!H^{0}(a,b),\,
            v_1 \in {}_0\!H^{1}(a,b)
        }
    }
    \nm{v_0}_{{}_0\!H^{0}(a,b)} + t \nm{v_1}_{{}_0\!H^{1}(a,b)}\\
    ={}&        K(t,\D_{a+}^{m-\gamma} \D_{a+}^{\beta}v),
    \end{split}
    \]
    for all $0<t<\infty$, which implies that
    \[
    \nm{\D_{a+}^{-\gamma} v}_{{}_0H^{\beta+\gamma}(a,b)}=
    \nm{\D_{a+}^{m-\gamma} \D_{a+}^{\beta}v}_{{}_0H^{\gamma-m}(a,b)}.
    \]
    Consequently, by the previous case, we have
    \[
        \nm{\D_{a+}^{-\gamma} v}_{{}_0H^{\beta+\gamma}(a,b)}
        \leqslant C_{\beta,\gamma}
            \nm{\D_{a+}^{\beta}v}_{{}_0H^{0}(a,b)}=
             C_{\beta,\gamma}
            \nm{v}_{{}_0H^{\beta}(a,b)}.
    \]
     This proves \cref{eq:regu-left} for the case $\beta\in\mathbb N$ and $m<\gamma<m+1,m\in\mathbb N$.

    Finally it remains to consider the case $\gamma\geqslant0$ and $n<\beta<n+1,n\in\mathbb N$. Since we have proved that
        \begin{align*}
    \nm{\D_{a+}^{-\gamma} w}_{{}_0H^{\gamma+n}(a,b)}  \leqslant{}&
C_{\beta,\gamma}\nm{w}_{{}_0H^n(a,b)} &&
\forall\,w\in {}_0H^n(a,b),\\
    \nm{\D_{a+}^{-\gamma} w}_{{}_0H^{\gamma+n+1}(a,b)}  \leqslant{}&
C_{\beta,\gamma}\nm{w}_{{}_0H^{n+1}(a,b)} &&
    \forall\,w \in {}_0H^{n+1}(a,b),
    \end{align*}
    applying the theory of interpolation spaces \cite[Lemma 22.3]{Tartar2007} gives
    \[
        \nm{\D_{a+}^{-\gamma} v}_{{}_0H^{\beta+\gamma}(a,b)}  \leqslant{}
    C_{\beta,\gamma} \nm{v}_{{}_0H^\beta(a,b)},
    \]
    for any $v \in {}_0H^\beta(a,b)$.
    This completes the proof of this lemma.
\end{proof}
\begin{rem}
In \cite[Theorem 2.1]{Gorenflo2015}, \cref{lem:regu-frac-int-leq} has been proved for $\beta=0$ and $0\leqslant \gamma\leqslant 1$.
\end{rem}
\begin{lem}
    \label{lem:regu-extra-D}
    Let $v\in L^2(a,b)$ and  $\beta\geqslant\gamma>0$.
    If $\D_{a+}^\gamma v\in {}_0H^{\beta-\gamma}(a,b)$, then
    \begin{equation}\label{eq:regu-D-left}
    \nm{v}_{{}_0H^\beta(a,b)}\leqslant
    C_{\beta,\gamma}  \nm{\D_{a+}^\gamma v}_{{}_0H^{\beta-\gamma}(a,b)}.
    \end{equation}
    If $\D_{b-}^\gamma v\in {}^0\!H^{\beta-\gamma}(a,b)$, then
    \begin{equation}\label{eq:regu-D-right}
    \nm{v}_{{}^0\!H^\beta(a,b)}\leqslant
    C_{\beta,\gamma}  \nm{\D_{b-}^\gamma v}_{{}^0\!H^{\beta-\gamma}(a,b)}.
    \end{equation}
\end{lem}
\begin{proof}
    Let us first prove \cref{eq:regu-D-left}. Suppose that $k<\gamma\leqslant k+1,~k\in\mathbb N$. By definition,
    \[
    \D_{a+}^\gamma v=\D^{k+1}\D_{a+}^{\gamma-k-1}v,
    \]
    then applying $\D_{a+}^{-k-1}$ on both sides of the above equation and using integral by parts yield that
    \begin{equation}\label{eq:e1}
(   \D_{a+}^{\gamma-k-1}v)(t)  = (  \D_{a+}^{-k-1}\D_{a+}^\gamma v)(t)
+
\sum_{i=0}^{k}
\frac{c_i(t-a)^i}{\Gamma(i+1)},
\quad a<t<b,
    \end{equation}
    where $c_i\in\mathbb R$.
    Moreover, since by \cref{lem:basic-frac}
    \[
    \begin{split}
    \D_{a+}^{\gamma}\D_{a+}^{k+1-\gamma}\D_{a+}^{\gamma-k-1}v =
        \D_{a+}^{\gamma}v ,\\
    \D_{a+}^{\gamma}\D_{a+}^{k+1-\gamma}\D_{a+}^{-k-1}\D_{a+}^\gamma v =
    \D_{a+}^{\gamma}\D_{a+}^{-\gamma}\D_{a+}^\gamma v=\D_{a+}^\gamma v,\\
    \end{split}
    \]
    applying $  \D_{a+}^{\gamma}\D_{a+}^{k+1-\gamma}$ on both sides of \cref{eq:e1} implies
    \[
    \sum_{i=0}^{k}
\frac{c_i(t-a)^i}{\Gamma(i-k)} = 0,
\quad a<t<b.
    \]
    Therefore, it follows that $c_i = 0$ for $0\leqslant i\leqslant k$,
    which, together with \cref{eq:e1}, gives
    \[
        \D_{a+}^{\gamma-k-1}v=  \D_{a+}^{-k-1}\D_{a+}^\gamma v.
    \]
    By \cref{lem:basic-frac}, applying $\D_{a+}^{k+1-\gamma}$ on both sides of the above equation yields that $v = \D_{a+}^{-\gamma}\D_{a+}^{\gamma}v$. Hence, by \cref{lem:regu-frac-int-leq},
    \[
        \nm{v}_{{}_0H^\beta(a,b)} =
    \nm{\D_{a+}^{-\gamma}\D_{a+}^{\gamma}v}_{{}_0H^\beta(a,b)}\leqslant
    C_{\beta,\gamma}  \nm{\D_{a+}^\gamma v}_{{}_0H^{\beta-\gamma}(a,b)},
    \]
    which proves \cref{eq:regu-D-left}. As \cref{eq:regu-D-right} can be proved similarly, this completes the proof.
\end{proof}
\begin{lem}
    \label{lem:compose-ID}
    If $\gamma>0$, then
    \begin{align}
        \D_{a+}^{-\gamma} \D_{a+}^\gamma v ={}& v\qquad
    \forall\,v \in {}_0H^\gamma(a,b),
    \label{eq:compose-ID-left}\\
        \D_{b-}^{-\gamma} \D_{b-}^\gamma v ={}& v\qquad
    \forall\,v \in {}^0\!H^\gamma(a,b).
    \label{eq:compose-ID-right}
    \end{align}
\end{lem}
\begin{proof}
    Let $ k \in \mathbb N_{>0} $ satisfy that $ k-1 < \gamma \leqslant k$. For any $ v \in
    {}_0H^\gamma(a,b) $, since \cref{lem:regu-frac-int-leq} implies $ \D_{a+}^{\gamma-k} v \in
    {}_0H^k(a,b) $, a straightforward computation yields that
    \begin{align*}
    \D_{a+}^{-\gamma} \D_{a+}^\gamma v =
    \D_{a+}^{-\gamma} \D^k \D_{a+}^{\gamma-k} v =
    \D^k \D_{a+}^{-\gamma} \D_{a+}^{\gamma-k} v =
    \D^k \D_{a+}^{-k} v = v,
    \end{align*}
    which proves \cref{eq:compose-ID-left}. An analogous argument proves \cref{eq:compose-ID-right} and thus concludes the proof of this lemma.
\end{proof}
\begin{lem}
    \label{lem:regu-frac-deriv}
    If $ \beta\geqslant \gamma>0$, then
        \begin{equation}    \label{eq:regu-frac-left-deriv}
    \small
    C_{\beta,\gamma} \nm{v}_{{}_0H^\beta(a,b)}\leqslant    \nm{\D_{a+}^\gamma v}_{{}_0H^{\beta-\gamma}(a,b)} \leqslant{}
    C_{\beta,\gamma} \nm{v}_{{}_0H^\beta(a,b)}
    \quad \forall \,v\in {}_0H^\beta(a,b),
    \end{equation}
    \begin{equation}    \label{eq:regu-frac-right-deriv}
    \small
    C_{\beta,\gamma} \nm{v}_{{}^0\!H^\beta(a,b)}\leqslant
    \nm{\D_{b-}^\gamma v}_{{}^0\!H^{\beta-\gamma}(a,b)} \leqslant{}
    C_{\beta,\gamma} \nm{v}_{{}^0\!H^\beta(a,b)}\quad
    \forall \,v\in {}^0\!H^\beta(a,b).
    \end{equation}
\end{lem}
\begin{proof}
    Since the proof of \cref{eq:regu-frac-right-deriv} is similar to that of \cref{eq:regu-frac-left-deriv}, we only
    prove \cref{eq:regu-frac-left-deriv}. If we can prove
    \begin{equation}\label{eq:leq}
    \nm{\D_{a+}^\gamma v}_{{}_0H^{\beta-\gamma}(a,b)} \leqslant{}
    C_{\beta,\gamma} \nm{v}_{{}_0H^\beta(a,b)}
    \end{equation}
    for all $ v\in {}_0H^\beta(a,b)$, then, by \cref{lem:regu-extra-D},
    \[
    \begin{split}
    {}& \nm{v}_{{}_0H^\beta(a,b)}
    \leqslant{}
    C_{\beta,\gamma}
    \nm{\D_{a+}^\gamma v}_{{}_0H^{\beta-\gamma}(a,b)}.
    \end{split}
    \]
    Hence it suffices to prove \cref{eq:leq}. By \cref{lem:basic-frac,lem:compose-ID},
    \[
    \D_{a+}^\gamma v = \D_{a+}^{\gamma} \D_{a+}^{-\beta}\D_{a+}^\beta v=\D_{a+}^{\gamma-\beta} \D_{a+}^\beta v,
    \]
    and using \cref{lem:regu-frac-int-leq} gives
    \[
    \begin{split}
    \nm{\D_{a+}^\gamma v}_{{}_0H^{\beta-\gamma}(a,b)} = {}&
    \nm{\D_{a+}^{\gamma-\beta} \D_{a+}^\beta v}_{{}_0H^{\beta-\gamma}(a,b)}
    \leqslant{}
    C_{\beta,\gamma} \nm{\D_{a+}^\beta v}_{{}_0H^0(a,b)}.
    \end{split}
    \]
    Let $k\in\mathbb N_{>0}$ satisfy that $k-1<\beta\leqslant k$. Invoking \cref{lem:regu-frac-int-leq} again implies that
    \[
    \begin{split}
    {}&    \nm{\D_{a+}^\beta v}_{{}_0H^{0}(a,b)}=
    \big\|\D^k\D_{a+}^{\beta-k} v\big\|_{{}_0H^{0}(a,b)}\\
    ={}&
    \big\|\D_{a+}^{\beta-k} v\big\|_{{}_0H^{k}(a,b)}
    \leqslant C_{\beta} \nm{v}_{{}_0H^\beta(a,b)},
    \end{split}
    \]
    which, together with the previous inequality, proves \cref{eq:leq}. This finishes the proof of this lemma.
\end{proof}

\begin{rem}
    In \cite[Theorem 2.1]{Gorenflo2015}, an alternative proof
    of \cref{eq:regu-frac-left-deriv} has been given for $0\leqslant \beta\leqslant 1$ and $\gamma=\beta$.
\end{rem}
\begin{lem}
    \label{lem:regu-frac-int-geq}
    If $  \beta,\gamma \geqslant0 $, then
    \begin{align*}
        \nm{v}_{{}_0H^\beta(a,b)}\leqslant {}&C_{\beta,\gamma}
    \nm{\D_{a+}^{-\gamma} v}_{{}_0H^{\beta+\gamma}(a,b)}   \quad
    \forall\,v \in {}_0H^\beta(a,b),\\
    \nm{v}_{{}^0\!H^\beta(a,b)}\leqslant {}&C_{\beta,\gamma}    \nm{\D_{b-}^{-\gamma} v}_{{}^0\!H^{\beta+\gamma}(a,b)}  \quad
    \forall\,v \in {}^0\!H^\beta(a,b).
    \end{align*}
\end{lem}
\begin{proof}
    Combining \cref{lem:compose-ID,lem:regu-frac-deriv}, we obtain
    \begin{align*}
        \nm{v}_{{}_0H^\beta(a,b)}\!=\!\nm{\D_{a+}^{\gamma}
        \D_{a+}^{-\gamma}v}_{{}_0H^\beta(a,b)}\!\leqslant {}&\!C_{\beta,\gamma} \! \nm{\D_{a+}^{-\gamma} v}_{{}_0H^{\beta+\gamma}(a,b)}   \quad
    \forall\,v \in {}_0H^\beta(a,b),\\
    \nm{v}_{{}^0\!H^\beta(a,b)}\!=\!\nm{\D_{b-}^{\gamma}\D_{b-}^{-\gamma}v}_{{}_0H^\beta(a,b)}\!\leqslant {}&\!C_{\beta,\gamma} \!   \nm{\D_{b-}^{-\gamma} v}_{{}^0\!H^{\beta+\gamma}(a,b)}  \quad
    \forall\,v \in {}^0\!H^\beta(a,b).
    \end{align*}
    This concludes the proof.
\end{proof}
\begin{lem}[\cite{Li2018A,Ervin2006}]
    \label{lem:coer}
      If $-1/2<\gamma<1/2$ and $ v\in H^{\max\{0,\gamma\}}(a,b) $, then
    \begin{align*}
    \cos(\gamma\pi) \nm{\D_{a+}^\gamma v}_{L^2(a,b)}^2\leqslant
    \dual{\D_{a+}^\gamma v, \D_{b-}^\gamma v}_{(a,b)}
    \leqslant \sec(\gamma\pi) \nm{\D_{a+}^\gamma v}_{L^2(a,b)}^2,
    \\
    \cos(\gamma\pi) \nm{\D_{b-}^\gamma v}_{L^2(a,b)}^2\leqslant
    \dual{\D_{a+}^\gamma v, \D_{b-}^\gamma v}_{(a,b)}
    \leqslant \sec(\gamma\pi) \nm{\D_{b-}^\gamma v}_{L^2(a,b)}^2.
    \end{align*}
    Moreover, if $v,w\in H^{\gamma}(a,b)$ with $0<\gamma<1/2$,
    then
    \[
    \big\langle\D_{a+}^{2\gamma} v, w\big\rangle_{H^\gamma(a,b)} =
    \dual{\D_{a+}^\gamma v, \D_{b-}^\gamma w}_{(a,b)} =
    \big\langle\D_{b-}^{2\gamma} w, v\big\rangle_{H^\gamma(a,b)}.
    \]
\end{lem}
\begin{lem}
    \label{lem:Nirenberg}
    If $   \beta \geqslant1$ and $ \gamma<1/2$, then
    \begin{equation}
    \label{eq:Nirenberg}
    \nm{v}_{C[0,1]} \leqslant C_{\beta,\gamma}
    \nm{v}_{{}_0H^\beta(0,1)}^{(1/2-\gamma)/(\beta-\gamma)}
    \nm{v}_{{}_0H^\gamma(0,1)}^{(\beta-1/2)/(\beta-\gamma)},
    \end{equation}
    for all $v \in {}_0H^\beta(0,1)$.
\end{lem}
\begin{proof}
    Let $ s:=\max\{0,\gamma\} $. Since $v\in {}_0H^\beta(0,1)$, by
    \cref{lem:basic-2,lem:compose-ID,lem:regu-frac-deriv,lem:coer}, a direct
    calculation gives
    \begin{small}
        \begin{align*}
        {}& \frac{1}{2} v^2(t) ={} \dual{v',v}_{(0,t)} =
        \dual{v',\D_{t-}^{-s}\D_{t-}^{s}v}_{(0,t)} =
        \dual{\D_{0+}^{-s}v',\D_{t-}^{s}v}_{(0,t)} \\
        ={} &
        \dual{\D_{0+}^{1-s}v,\D_{t-}^{s}v}_{(0,t)} \leqslant
        \nm{\D_{0+}^{1-s}v}_{L^2(0,t)}
        \nm{\D_{t-}^{s}v}_{L^2(0,t)} \\
        \leqslant {}&
        C_s\nm{\D_{0+}^{1-s}v}_{L^2(0,t)}
        \nm{\D_{0+}^{s}v}_{L^2(0,t)}
        \leqslant C_s\nm{\D_{0+}^{1-s}v}_{L^2(0,1)}
        \nm{\D_{0+}^{s}v}_{L^2(0,1)} \\
        \leqslant{} &
        C_s \nm{v}_{{}_0H^{1-s}(0,1)}
        \nm{v}_{{}_0\!H^s(0,1)},
        \end{align*}
    \end{small}
    for all $0\leqslant t\leqslant 1$. It follows that
    \[
    \nm{v}_{C[0,1]} \leqslant
    C_s \nm{v}^{1/2}_{{}_0H^s(0,1)}
    \nm{v}^{1/2}_{{}_0H^{1-s}(0,1)}.
    \]
    Additionally, by \cref{lem:inter_space} and \cite[Corollary~1.7]{Lunardi2018} we
    have
    \begin{align*}
    \nm{v}_{{}_0H^s(0,1)} & \leqslant C_{\beta,\gamma}
    \nm{v}_{{}_0H^\beta(0,1)}^{(s-\gamma)/(\beta-\gamma)}
    \nm{v}_{{}_0H^\gamma(0,1)}^{(\beta-s)/(\beta-\gamma)}, \\
    \nm{v}_{{}_0H^{1-s}(0,1)} & \leqslant C_{\beta,\gamma}
    \nm{v}_{{}_0H^\beta(0,1)}^{(1-\gamma-s)/(\beta-\gamma)}
    \nm{v}_{{}_0H^\gamma(0,1)}^{(\beta+s-1)/(\beta-\gamma)}.
    \end{align*}
    Consequently, combining the above three estimates proves \cref{eq:Nirenberg}.
\end{proof}
\begin{lem}\label{lem:trace}
    If $v\in {}_0H^{\beta}(0,1)$ with $\beta>1/2$, then
    \begin{equation}\label{eq:trace}
    \nm{v}_{C[0,1]}\leqslant
    \frac{C_{\beta}}{\epsilon}
        \nm{v}_{{}_0H^{1/2}(0,1)}^{1-\epsilon}
    \nm{v}_{{}_0H^{\beta}(0,1)}^{\epsilon},
    \end{equation}
    for all $0<\epsilon\leqslant1/\max\{2,2\beta \}$.
\end{lem}
\begin{proof}
    For any $w\in L^2(0,1)$, extend $w$ to $(-\infty,0)$ by zero and denote this extension by $\widetilde{w}$. Let $n\in\mathbb N$ satisfy that $n-1<\beta\leqslant n$. Following the proof of \cite[Theorem 8.1]{Lions-I}, we define an extension operator $ E: L^2(0,1) \to L^2(\mathbb R) $
    by that, for any $ w \in L^2(0,1) $,
    \[
    (Ew)(t): =
    \left\{
    \begin{aligned}
    &\widetilde{w}(t),&&{\rm if~}t<1,\\
    &    \sum_{j=1}^{n+1}
    \gamma_j\widetilde{w}(j+1-jt),&&{\rm if~}t>1,
    \end{aligned}
    \right.
    \]
    where $\gamma_j$ is defined by
    \[
    \sum_{j=1}^{n+1}(-j)^k\gamma_j = 1,\quad0\leqslant k\leqslant n.
    \]
    Since a straightforward computation gives
    \[
    \begin{aligned}
    \nm{Ew}_{H^k(\mathbb R)} \leqslant {}&C_k \nm{w}_{{}_0H^k(0,1)}
    &&\forall\,w\in{}_0H^k(0,1),\quad0\leqslant k\leqslant n,
    \end{aligned}
    \]
    applying \cite[Lemma~22.3]{Tartar2007} yields
    \[
    \begin{aligned}
    \nm{Ev}_{H^{(1-\epsilon)/2+\epsilon\beta}(\mathbb R)}
    \leqslant {}&   C_{\beta}
    \nm{v}_{[{}_0H^{1/2}(0,1),\,{}_0H^\beta(0,1)]_{\epsilon,2}}.
    &&
    \end{aligned}
    \]
    In addition, \cite[(23.11)]{Tartar2007} implies that
    \begin{equation*}
    \Big(
    \int_\mathbb R
    \big(1+\snm{\xi}^2\big)^{(1-\epsilon)/2+\epsilon\beta}
    \snm{(\mathcal FEv)(\xi)}^2 \, \mathrm{d}\xi
    \Big)^{1/2} \leqslant
    C_\beta  \nm{Ev}_{H^{(1-\epsilon)/2+\epsilon\beta}(\mathbb R)},
    \end{equation*}
    where $ \mathcal F $ is the Fourier transform operator. Moreover, using \cite[(1.2.4) and
    (1.2.5)]{Agranovich2015} yields
    \begin{equation*}
    \nm{Ev}_{L^\infty(\mathbb R)}
    \leqslant \frac {C_\beta}{\sqrt{\epsilon}}
    \Big(
    \int_\mathbb R \big(1+\xi^2\big)^{(1-\epsilon)/2+\epsilon\beta}
    \snm{(\mathcal FEv)(\xi)}^2
    \, \mathrm{d}\xi
    \Big)^{1/2}.
    \end{equation*}
    Therefore it follows that
    \[
    \begin{split}
    \nm{v}_{C[0,1]}=    \nm{Ev}_{L^\infty(\mathbb R)}
    \leqslant \frac {C_\beta}{\sqrt{\epsilon}}
    \nm{v}_{[{}_0H^{1/2}(0,1),\,{}_0H^\beta(0,1)]_{\epsilon,2}}.
    \end{split}
    \]
    Since borrowing the proof of \cite[Corollary 1.7]{Lunardi2018} gives
    \[
    \nm{v}_{[{}_0H^{1/2}(0,1),\,{}_0H^\beta(0,1)]_{\epsilon,2}}
    \leqslant \frac{1}{\sqrt{\epsilon}}
    \nm{v}_{{}_0H^{1/2}(0,1)}^{1-\epsilon}
    \nm{v}_{{}_0H^{\beta}(0,1)}^{\epsilon},
    \]
    we finally obtain \cref{eq:trace} by the above two inequalities. This concludes the proof of this lemma.
\end{proof}

Now let us generalize the fractional integral operator as follows. Recall that in this section $ X $ denotes a separable Hilbert space.
Assume that $\beta,\gamma>0$ and $ v \in{}_0H^{-\beta}(a,b;X) $. If $0<\gamma\leqslant\beta$, then define
$\D_{a+}^{-\gamma} v \in {}_0H^{\gamma-\beta}(a,b;X) $
by that
\begin{equation}\label{eq:def-int-1}
\dual{\D_{a+}^{-\gamma} v, w}_{{}^0\!H^{\beta-\gamma}(a,b;X)} :=
\dual{ v, \D_{b-}^{-\gamma} w}_{{}^0\!H^{\beta}(a,b;X)},
\end{equation}
for all $w\in{}^0\!H^{\beta-\gamma}(a,b;X)$; if $\gamma>\beta$, then define $\D_{a+}^{-\gamma} v \in {}_0H^{\gamma-\beta}(a,b;X) $ by that
\begin{equation}\label{eq:def-int-2}
\D_{a+}^{-\gamma} v :=
\D_{a+}^{\beta-\gamma} \D_{a+}^{-\beta} v
\end{equation}
By \cref{lem:regu-frac-int-leq,lem:regu-frac-deriv,lem:regu-frac-int-geq}, the generalized left-sided fractional integral operator
\[
\D_{a+}^{-\gamma}: {}_0H^{-\beta}(a,b;X)\to{}_0H^{\gamma-\beta}(a,b;X)
\]
is well-defined for all $\beta,\gamma>0$. Symmetrically,
we can generalize the right-sided fractional integral operator as follows. Assume that $\beta,\gamma>0$ and $ v \in{}^0\!H^{-\beta}(a,b;X) $. Define
$\D_{b-}^{-\gamma} v\in{}^0\!H^{\gamma-\beta}(a,b;X)$ by that
\[
\dual{ \D_{b-}^{-\gamma} v,w}_{{}_0H^{\beta-\gamma}(a,b;X)}:=
\dual{v,\D_{a+}^{-\gamma}w}_{{}_0H^{\beta}(a,b;X)},
\]
for all $w\in{}_0H^{\beta-\gamma}(a,b;X)$.
\begin{lem}
    \label{lem:regu-frac-int-general}
    If $ \beta,\gamma>0$, then
        \begin{equation}    \label{eq:regu-general-left}
    \small
    C_{\beta,\gamma} \nm{v}_{{}_0H^{-\beta}(a,b)}\leqslant
    \nm{\D_{a+}^{-\gamma} v}_{{}_0H^{\gamma-\beta}(a,b)} \leqslant{}
    C_{\beta,\gamma} \nm{v}_{{}_0H^{-\beta}(a,b)}
    \quad \forall \,v\in {}_0H^{-\beta}(a,b),
    \end{equation}
    \begin{equation}    \label{eq:regu-general-right}
    \small
    C_{\beta,\gamma} \nm{v}_{{}^0\!H^{-\beta}(a,b)}\leqslant
    \nm{\D_{b-}^{-\gamma} v}_{{}^0\!H^{\gamma-\beta}(a,b)} \leqslant{}
    C_{\beta,\gamma} \nm{v}_{{}^0\!H^{-\beta}(a,b)}\quad
    \forall \,v\in {}^0\!H^{-\beta}(a,b).
    \end{equation}
\end{lem}
\begin{proof}
    Since the proofs of \cref{eq:regu-general-left,eq:regu-general-right} are similar, we only give the proof of the former.

    Let us first prove that
    \begin{equation}\label{eq:tmp}
    \nm{\D_{a+}^{-\gamma} v}_{{}_0H^{\gamma-\beta}(a,b)} \leqslant{}
    C_{\beta,\gamma} \nm{v}_{{}_0H^{-\beta}(a,b)}
    \quad \forall \,v\in {}_0H^{-\beta}(a,b).
    \end{equation}
    If $\gamma\leqslant\beta $, then by \cref{lem:regu-frac-int-leq} and definition \cref{eq:def-int-1},
    \[
    \begin{split}
    {}&\dual{\D_{a+}^{-\gamma} v, w}_{{}^0\!H^{\beta-\gamma}(a,b)} ={}
    \dual{ v, \D_{b-}^{-\gamma} w}_{{}^0\!H^{\beta}(a,b)}\\
    \leqslant{}&
    \nm{v}_{{}_0H^{-\beta}(a,b)}
    \nm{\D_{b-}^{-\gamma} w}_{{}^0\!H^{\beta}(a,b)}\\
    \leqslant{}&C_{\beta,\gamma}
    \nm{v}_{{}_0H^{-\beta}(a,b)}
    \nm{w}_{{}^0\!H^{\beta-\gamma}(a,b)},
    \end{split}
    \]
    for all $w\in{}^0\!H^{\beta-\gamma}(a,b)$, which proves \cref{eq:tmp} for $\gamma\leqslant\beta $. If $\gamma>\beta$, then by definition \cref{eq:def-int-2} and \cref{lem:regu-frac-int-leq} and the previous case, we have
    \[
    \begin{split}
    {}& \nm{\D_{a+}^{-\gamma} v}_{{}_0H^{\gamma-\beta}(a,b)} =
    \big\|\D_{a+}^{\beta-\gamma} \D_{a+}^{-\beta} v\big\|_{{}_0H^{\gamma-\beta}(a,b)}\\
    \leqslant{}&C_{\beta,\gamma}
    \big\|\D_{a+}^{-\beta} v\big\|_{L^{2}(a,b)}
    \leqslant{}C_{\beta,\gamma}
    \nm{v}_{{}_0H^{-\beta}(a,b)}.
    \end{split}
    \]
    This proves \cref{eq:tmp} for the case $\gamma>\beta$.

    Then it remains to prove that
    \begin{equation}\label{eq:tmp1}
    \nm{v}_{{}_0H^{-\beta}(a,b)}\leqslant{}
    C_{\beta,\gamma}    \nm{\D_{a+}^{-\gamma} v}_{{}_0H^{\gamma-\beta}(a,b)}
    \quad \forall \,v\in {}_0H^{-\beta}(a,b).
    \end{equation}
    If $\gamma\leqslant\beta $, then by definition \cref{eq:def-int-1} and \cref{lem:compose-ID,lem:regu-frac-deriv}
    \[
    \begin{split}
    \dual{ v, w}_{{}^0\!H^{\beta}(a,b)} =       {}&\dual{ v, \D_{b-}^{-\gamma} \D_{b-}^{\gamma} w}_{{}^0\!H^{\beta}(a,b)}
    ={}
    \dual{\D_{a+}^{-\gamma} v, \D_{b-}^{\gamma} w}_{{}^0\!H^{\beta-\gamma}(a,b)}    \\
    \leqslant{}&
    \nm{\D_{a+}^{-\gamma} v}_{{}_0H^{\gamma-\beta}(a,b)}
    \nm{\D_{b-}^{-\gamma} w}_{{}^0\!H^{\beta-\gamma}(a,b)}\\
    \leqslant{}&C_{\beta,\gamma}
    \nm{\D_{a+}^{-\gamma} v}_{{}_0H^{\gamma-\beta}(a,b)}
    \nm{w}_{{}^0\!H^{\beta}(a,b)},
    \end{split}
    \]
    for all $w\in{}^0\!H^{\beta}(a,b)$, so that we have
    \[
    \nm{v}_{{}_0H^{-\beta}(a,b)}\leqslant C_{\beta,\gamma}
    \nm{\D_{a+}^{-\gamma} v}_{{}_0H^{\gamma-\beta}(a,b)}.
    \]
    If $\gamma>\beta$, then by definition \cref{eq:def-int-2} and  \cref{lem:compose-ID,lem:regu-frac-deriv}, an evident calculation gives
    \[
    \begin{split}
    {}&\dual{ v, w}_{{}^0\!H^{\beta}(a,b)} =        \big\langle v, \D_{b-}^{-\beta} \D_{b-}^{\beta} w\big\rangle_{{}^0\!H^{\beta}(a,b)}\\
    ={}&
    \big\langle\D_{a+}^{-\beta} v, \D_{b-}^{\beta} w\big\rangle_{(a,b)}     =
    \big\langle\D_{a+}^{\gamma-\beta} \D_{a+}^{-\gamma} v, \D_{b-}^{\beta} w\big\rangle_{(a,b)}     \\
    \leqslant{}&
    \big\|\D_{a+}^{\gamma-\beta} \D_{a+}^{-\gamma} v\big\|_{L^{2}(a,b)}
    \big\|\D_{b-}^{\beta} w\big\|_{L^2(a,b)}\\
    \leqslant{}&C_{\beta,\gamma}
    \nm{\D_{a+}^{-\gamma} v}_{{}_0H^{\gamma-\beta}(a,b)}
    \nm{w}_{{}^0\!H^{\beta}(a,b)},
    \end{split}
    \]
    for all $w\in{}^0\!H^{\beta}(a,b)$, which implies that
    \[
    \nm{v}_{{}_0H^{-\beta}(a,b)}\leqslant C_{\beta,\gamma}
    \nm{\D_{a+}^{-\gamma} v}_{{}_0H^{\gamma-\beta}(a,b)}.
    \]
    This proves \cref{eq:tmp1} and thus completes the proof of this lemma.
\end{proof}
\begin{lem}\label{lem:dual-ID}
    If $v\in H^{\gamma}(a,b) $ with $0<\gamma<1/2$, then
    \begin{equation}\label{eq:dualID}
    \big|\big\langle
    \D_{a+}^{ 2\gamma} v,
    \D_{a+}^{-2\gamma} v
    \big\rangle_{H^{\gamma}(a,b)}\big|
    \leqslant C_\gamma
    \nm{v}_{{}_0H^{ \gamma}(a,b)}
    \nm{v}_{{}_0H^{-\gamma}(a,b)}.
    \end{equation}
\end{lem}
\begin{proof}
    Since by \cref{lem:coer},
    \[
    \big\|\D_{b-}^{\gamma}\D_{a+}^{-2\gamma} v\big\|_{L^{2}(a,b)}
    \leqslant C_\gamma
    \big\|\D_{a+}^{\gamma}\D_{a+}^{-2\gamma} v\big\|_{L^{2}(a,b)}
    =C_\gamma
    \big\|\D_{a+}^{-\gamma} v\big\|_{L^{2}(a,b)},
    \]
    it follows that
    \[
    \begin{split}
    {}&\big|\big\langle
    \D_{a+}^{ 2\gamma} v,
    \D_{a+}^{-2\gamma} v
    \big\rangle_{H^{\gamma}(a,b)}\big| =
    \big|  \big\langle
    \D_{a+}^{ \gamma} v,\D_{b-}^{\gamma}\D_{a+}^{-2\gamma} v
    \big\rangle_{(a,b)}\big|  \\
    \leqslant{}&
    \big\|\D_{a+}^{ \gamma} v\big\|_{L^{2}(a,b)}
    \big\|\D_{b-}^{\gamma}\D_{a+}^{-2\gamma} v\big\|_{L^{2}(a,b)}\\
    \leqslant{}&C_\gamma
    \big\|\D_{a+}^{ \gamma}v\big\|_{L^{2}(a,b)}
    \big\|\D_{a+}^{-\gamma}v\big\|_{L^{2}(a,b)}\\
    \leqslant {}&C_\gamma
    \nm{v}_{{}_0H^{ \gamma}(a,b)}
    \nm{v}_{{}_0H^{-\gamma}(a,b)},
    \end{split}
    \]
    by \cref{lem:regu-frac-deriv,lem:regu-frac-int-general}. This completes the proof.
\end{proof}
\begin{lem}\label{lem:3.13}
    Assume that $\gamma<\beta+1/2$. If $v\in{}_0H^{\beta}(a,b)$, then
    $\D_{a+}^\gamma v\in{}_0H^{\beta-\gamma}(a,b)$ and
    \begin{equation}\label{eq:Dvw-1}
    \dual{\D_{a+}^\gamma v,w}_{{}^0\!H^{\gamma-\beta}(a,b)} =   \dual{\D_{a+}^\beta v,\D_{b-}^{\gamma-\beta} w}_{(a,b)}
    \end{equation}
    for all $w\in{}^0\!H^{\gamma-\beta}(a,b)$.
    If $v\in{}^0\!H^{\beta}(a,b)$,
    then $\D_{b-}^\gamma v\in{}^0\!H^{\beta-\gamma}(a,b)$ and
    \begin{equation}\label{eq:Dvw-2}
    \dual{\D_{b-}^\gamma v,w}_{{}_0H^{\gamma-\beta}(a,b)} =     \dual{\D_{b-}^\beta v,\D_{b-}^{\gamma-\beta} w}_{(a,b)}
    \end{equation}
    for all $w\in{}_0H^{\gamma-\beta}(a,b)$.
\end{lem}
\begin{proof}
    As the proof of \cref{eq:Dvw-2} is similar to that of \cref{eq:Dvw-1}, we only prove \cref{eq:Dvw-1}.

    Firstly, if $0<\gamma\leqslant\beta$, then by \cref{lem:regu-frac-deriv}, it is obvious that $\D_{a+}^\gamma v\in{}_0H^{\beta-\gamma}(a,b)$, and  \cref{eq:Dvw-1} holds indeed by by \cref{lem:compose-ID} and definition \cref{eq:def-int-1}.

    Next we consider the case $\gamma\leqslant0$ and $ \beta\geqslant\gamma$. By \cref{lem:regu-frac-int-leq,lem:regu-frac-int-general}, we have $\D_{a+}^\gamma v\in{}_0H^{\beta-\gamma}(a,b)$. If $\beta\geqslant 0$, then by \cref{lem:basic-frac,lem:compose-ID}, it is evident that
    \[
    \D_{a+}^{\gamma-\beta} \D_{a+}^\beta v=
    \D_{a+}^{\gamma} \D_{a+}^{-\beta}\D_{a+}^{\beta}  v=
    \D_{a+}^{\gamma} v.
    \]
    If $\beta<0$, then by definition \cref{eq:def-int-2},
    \[
    \begin{split}
{}& \dual{\D_{b-}^{\gamma-\beta} w,\D_{a+}^\beta v}_{(a,b)}=
    \dual{w,\D_{a+}^{\gamma-\beta} \D_{a+}^\beta v}_{{}_0H^{\beta-\gamma}(a,b)}\\
    ={}&
    \dual{w,\D_{a+}^{\gamma}  v}_{{}_0H^{\beta-\gamma}(a,b)}=
    \dual{\D_{a+}^\gamma v,w}_{{}^0\!H^{\gamma-\beta}(a,b)},
    \end{split}
    \]
    for all $w\in{}^0\!H^{\gamma-\beta}(a,b)$, which proves \cref{eq:Dvw-1} for $\gamma\leqslant0$ and $ \beta\geqslant\gamma$.

    Then let us consider the case $\gamma\leqslant0$ and $ \gamma-1/2<\beta<\gamma$. By \cref{lem:regu-frac-int-general}, we have $\D_{a+}^\gamma v\in{}_0H^{\beta-\gamma}(a,b)$, and using \cref{lem:basic-frac,lem:compose-ID} and definition \cref{eq:def-int-1} gives
    \[
    \begin{split}
    {}&\dual{\D_{a+}^\gamma v,w}_{{}^0\!H^{\gamma-\beta}(a,b)} \\
    = {}&   \dual{v,\D_{b-}^{\gamma} w}_{{}^0\!H^{-\beta}(a,b)}
    ={} \dual{v,\D_{b-}^{\gamma} \D_{b-}^{\beta-\gamma}
        \D_{b-}^{\gamma-\beta} w}_{{}^0\!H^{-\beta}(a,b)}\\
    ={}&\dual{v,\D_{b-}^{\beta}
        \D_{b-}^{\gamma-\beta} w}_{{}^0\!H^{-\beta}(a,b)}
    ={}\dual{\D_{a+}^{\beta}v,
        \D_{b-}^{\gamma-\beta} w}_{(a,b)},
    \end{split}
    \]
    for all $w\in{}^0\!H^{\gamma-\beta}(a,b)$. This proves \cref{eq:Dvw-1} for $\gamma\leqslant0$ and $ \gamma-1/2<\beta<\gamma$.

    Finally it remains to consider the case $\gamma>0$ and $\gamma-1/2<\beta<\gamma$ . Let $k\in\mathbb N$ satisfy $k-1<\gamma\leqslant k$. If $\beta\geqslant0$, then by \cref{lem:basic-2,lem:regu-frac-deriv}, a direct manipulation implies that
    \[
    \begin{split}
    \dual{\D_{a+}^\gamma v,\phi} = {}&  \dual{\D^k\D_{a+}^{\gamma-k} v,\phi} =
    (-1)^k\dual{\D_{a+}^{\gamma-k} v,\D^k\phi}_{(a,b)} \\
    ={}&        (-1)^k\dual{ v,\D_{b-}^{\gamma-k}\D^k\phi}_{ (a,b)}
    ={} \dual{ v,\D_{b-}^{\gamma}\phi}_{(a,b)} \\
    ={}& \dual{ v,\D_{b-}^{\beta}\D_{b-}^{\gamma-\beta}\phi}_{ (a,b)}
    = {}\dual{ \D_{a+}^{\beta}v,\D_{b-}^{\gamma-\beta}\phi}_{(a,b)}\\
    \leqslant {}&C_{\beta,\gamma}\nm{v}_{{}_0H^\beta(a,b)}
    \nm{\phi}_{{}^0\!H^{\gamma-\beta}(a,b)},
    \end{split}
    \]
    for all $\phi\in C_0^\infty(a,b)$.
    If $\beta<0$, then $0<\gamma<1/2$. By \cref{lem:regu-frac-deriv,lem:regu-frac-int-general} and definition \cref{eq:def-int-2}, for all $\phi\in C_0^\infty(a,b)$, a similar deduction gives
    \[
    \begin{split}
    {}&\dual{\D_{a+}^\gamma v,\phi}
    =   \dual{\D\D_{a+}^{\gamma-1} v,\phi} \\
    = {}&
    -\dual{\D_{a+}^{\gamma-1} v,\D\phi}_{(a,b)} =
    -\dual{ \D_{a+}^{\gamma-\beta-1} \D_{a+}^{\beta}v,\D\phi}_{ (a,b)}
    \\
    ={}&
    -\dual{ \D_{a+}^{\beta}v,\D_{b-}^{\gamma-\beta-1} \D\phi}_{ (a,b)}
    = \dual{ \D_{a+}^{\beta}v,\D_{b-}^{\gamma-\beta} \phi}_{ (a,b)} \\
    \leqslant {}&C_{\beta,\gamma}\nm{v}_{{}_0H^\beta(a,b)}
    \nm{\phi}_{{}^0\!H^{\gamma-\beta}(a,b)}.
    \end{split}
    \]
    Above, we use $\dual{\cdot,\cdot}$ to denote the dual pair between the dual space of $C_{0}^\infty(a,b)$ and $C_{0}^\infty(a,b)$. Since $0<\gamma-\beta<1/2$, it is clear that $C_0^\infty(a,b)$ is dense in ${}^0\!H^{\gamma-\beta}(a,b)$. Consequently, we conclude that $\D_{a+}^\gamma v\in{}_0H^{\beta-\gamma}(a,b)$ and \cref{eq:Dvw-1} holds indeed. This concludes the proof of this lemma.
\end{proof}
\begin{lem}\label{lem:DD}
    If $\max\{\beta,\,\beta+\gamma\}<s+1/2$, then
    \begin{align}
\D_{a+}^\gamma  \D_{a+}^\beta v = {}&
    \D_{a+}^{\beta+\gamma}v\quad\forall\,v\in{}_0H^{s}(a,b),
    \label{eq:DD1}\\
    \D_{b-}^\gamma  \D_{b-}^\beta v = {}&
    \D_{b-}^{\beta+\gamma}v\quad\forall\,v\in{}^0\!H^{s}(a,b).
    \label{eq:DD2}
    \end{align}
\end{lem}
\begin{proof}
    Let us first prove that
    \begin{equation}\label{eq:DD11}
    \D_{a+}^{s-\beta}\D_{a+}^\beta v = \D_{a+}^s v
    \quad\forall\,v\in{}_0H^{s}(a,b).
    \end{equation}
    By \cref{lem:regu-frac-deriv,lem:regu-frac-int-general}, it is evident that
    $\D_{a+}^{s}v\in L^2(a,b)$. In addition, applying \cref{lem:3.13} yields that
    \[
    \begin{split}
    {}& \dual{  \D_{a+}^{s-\beta}\D_{a+}^\beta v, \phi}_{(a,b)} =
    \dual{\D_{a+}^\beta v, \D_{b-}^{s-\beta}\phi}_{{}^0\!H^{\beta-s}(a,b)}\\
    = {}&
    \dual{  \D_{a+}^{s} v,\D_{b-}^{\beta-s} \D_{b-}^{s-\beta}\phi}_{ (a,b)}=\dual{  \D_{a+}^{s} v, \phi}_{ (a,b)},
    \end{split}
    \]
    for any $\phi\in C_0^\infty(a,b)$, which proves \cref{eq:DD11}.

    Then we turn to the proof of \cref{eq:DD1}.
    Since using \cref{lem:3.13} implies that
    $\D_{a+}^\beta v\in{}_0H^{s-\beta}(a,b) $ and both $\D_{a+}^\gamma\D_{a+}^\beta v $ and $\D_{a+}^{\beta+\gamma}v$ belong to ${}_0H^{s-\gamma-\beta}(a,b)$,
    by \cref{eq:DD11}, we have
    \[
    \begin{split}
    {}& \dual{\D_{a+}^\gamma\D_{a+}^\beta v, w}_{{}^0\!H^{\beta+\gamma-s}(a,b)} =
    \dual{\D_{a+}^{s-\beta}\D_{a+}^\beta v, \D_{b-}^{\beta+\gamma-s}w}_{(a,b)}\\
    = {}&
    \dual{\D_{a+}^s v, \D_{b-}^{\beta+\gamma-s}w}_{(a,b)}=
    \dual{\D_{a+}^{\beta+\gamma} v, w}_{{}^0\!H^{\beta+\gamma-s}(a,b)},
    \end{split}
    \]
    for all $w\in{}^0\!H^{\beta+\gamma-s}(a,b)$, which proves \cref{eq:DD1}. As \cref{eq:DD2} can be proved analogously, we finish the proof of this lemma.
\end{proof}
\section{Weak Solution and Regularity}
\subsection{The first definition}
\label{sec:regu}
Define
\begin{align*}
  \widehat W &:= {}^0\!H^{3\alpha/4}(0,T;L^2(\Omega))
\cap {}^0\!H^{\alpha/4}(0,T;\dot H^1(\Omega)),\\
  W &:= {}_0H^{\alpha/4}(0,T;L^2(\Omega))
  \cap {}_0H^{-\alpha/4}(0,T;\dot H^1(\Omega)),
\end{align*}
and endow these two spaces with the norms
\begin{align*}
  \nm{\cdot}_{\widehat W} &:= \left(
\nm{\cdot}^2_{{}^0\!H^{3\alpha/4}(0,T;L^2(\Omega))} +
\nm{\cdot}^2_{{}^0\!H^{\alpha/4}(0,T;\dot H^1(\Omega))}
\right)^{1/2},\\
  \nm{\cdot}_W &:= \left(
    \nm{\cdot}^2_{{}_0H^{\alpha/4}(0,T;L^2(\Omega)) } +
    \nm{\cdot}^2_{{}_0H^{-\alpha/4}(0,T;\dot H^1(\Omega))}
  \right)^{1/2},
\end{align*}
respectively. Assuming that
\[
 \D_{0+}^{\alpha/2}(u_0+tu_1)\in W^*
\quad{\rm and}\quad f\in \widehat W^*,
\]
we call $ u \in W $ a weak
solution to problem \cref{eq:model} if
\begin{equation}
\label{eq:weak_sol}
\begin{split}
&{}\big\langle \D_{0+}^{\alpha/2} u, v\big\rangle_{{}^0\!H^{\alpha/4}(0,T;L^2(\Omega)) } +
\big\langle
\nabla \D_{0+}^{-\alpha/4}u, \nabla \D_{T-}^{-\alpha/4} v
\big\rangle_{ \Omega\times(0,T)} \\
={}&
\big\langle f, \D_{T-}^{-\alpha/2}v\big\rangle_{\widehat W}
+\big\langle \D_{0+}^{\alpha/2}(u_0+tu_1), v\big\rangle_{W}
\end{split}
\end{equation}
for all $ v \in W $.
     \begin{rem}
    Notice that
    \begin{align*}
    {}_0H^{\alpha/4}(0,T;L^2(\Omega)) &=
    {}^0\!H^{\alpha/4}(0,T;L^2(\Omega)), \\
    {}_0H^{-\alpha/4}(0,T;\dot H^1(\Omega)) &=
    {}^0\!H^{-\alpha/4}(0,T;\dot H^1(\Omega)),
    \end{align*}
    with equivalent norms. Hence, by \cref{lem:regu-frac-int-leq,lem:coer}, it is easy
    to verify that each term in \cref{eq:weak_sol} makes sense.
\end{rem}
By \cref{lem:regu-frac-int-leq,lem:regu-frac-deriv,lem:coer} and the well-known Lax-Milgram theorem, a
routine argument yields that the above weak solution is well-defined.
\begin{thm}
  \label{thm:basic_weak_sol}
  Problem \cref{eq:weak_sol} admits a unique solution $ u \in W $ and
  \[
    \nm{u}_W \leqslant
    C_\alpha
    \left(\nm{f}_{\widehat W^*}+\big\| \D_{0+}^{\alpha/2}(u_0+tu_1)\big\|_{W^*}\right).
  \]
\end{thm}
\begin{rem}
    If $u_0\in L^2(\Omega)$ and $u_1\in \dot H^{-1}(\Omega)$, then a simple calculation
    yields that $    \D_{0+}^{\alpha/2}(u_0+tu_1)\in W^*$.
    Therefore the weak solution is well-defined by \cref{eq:weak_sol} for $u_0\in L^2(\Omega),u_1\in \dot H^{-1}(\Omega)$ and $ f\in \widehat W^*$.
\end{rem}

Next, we employ the Galerkin method to investigate the regularity of the solution to
problem \cref{eq:weak_sol} in the case $u_0=u_1=0$.
Let us first define $ y \in {}_0H^{\alpha/4}(0,T) $ by that
\begin{equation}
\label{eq:frac_ode_weak}
\big\langle\D_{0+}^{\alpha/2 }y,z  \big\rangle_{ {}^0\!H^{\alpha/4}(0,T)} +
\lambda   \big\langle\D_{0+}^{-\alpha/2}y,z  \big\rangle_{(0,T)}
=   \big\langle\D_{0+}^{-\alpha/2}g,z  \big\rangle_{{}^0\!H^{\alpha/4}(0,T)}
\end{equation}
for all $ z \in {}_0H^{\alpha/4}(0,T) $, where $ g \in{}_0H^{-3\alpha/4}(0,T) $ and $
\lambda $ is a positive constant. Similar to problem \cref{eq:weak_sol}, problem
\cref{eq:frac_ode_weak} admits a unique solution $ y \in {}_0H^{\alpha/4}(0,T) $ and
\begin{equation*}
  \nm{y}_{{}_0H^{\alpha/4}(0,T)} +
  \lambda^{1/2} \nm{y}_{{}_0H^{-\alpha/4}(0,T)}
  \leqslant C_\alpha \nm{g}_{ {}_0H^{-3\alpha/4}(0,T) }.
\end{equation*}

\begin{lem}
    \label{lem:regu-ode}
    If $ \gamma\geqslant-3\alpha/4$ and $ g \in {}_0H^{\gamma}(0,T)$, then
    \begin{equation}
    \label{eq:ode-strong}
    \D_{0+}^{\alpha+\gamma} y + \lambda \D_{0+}^\gamma y =
    \D_{0+}^\gamma g,
    \end{equation}
    \begin{equation}
    \label{eq:regu-ode}
    \nm{y}_{{}_0H^{\alpha+\gamma}(0,T)} +
    \lambda\nm{y}_{{}_0H^\gamma(0,T)}
    \leqslant C_{\alpha,\gamma}
    \nm{g}_{{}_0H^{\gamma}(0,T)}.
    \end{equation}
    Moreover, if $ 1-\alpha \leqslant \gamma<1/2 $ then
    \begin{equation}
    \label{eq:regu-ode-C}
    \lambda^{1+(\gamma-1/2)/\alpha} \nm{y}_{C[0,T]}
    \leqslant C_{\alpha,\gamma,T} \nm{g}_{{}_0H^\gamma(0,T)},
    \end{equation}
    and if $\gamma=1/2$ then
        \begin{equation}
    \label{eq:regu-ode-C-limit-case}
    \lambda^{1-\epsilon/2} \nm{y}_{C[0,T]}
    \leqslant       \frac{C_{\alpha,T}}{
        \epsilon} \nm{g}_{{}_0H^{1/2}(0,T)},
    \end{equation}
    for all $0<\epsilon\leqslant2/(2\alpha+1)$.
\end{lem}
\begin{proof}
    Firstly, let us consider the following problem: find $ w \in {}_0H^{\alpha/4}(0,T)$ such that
    \begin{equation}\label{eq:weak-w}
    \big\langle\D_{0+}^{\alpha/2} w, z  \big\rangle_{{}^0\!H^{\alpha/4}(0,T)} +
    \lambda \big\langle\D_{0+}^{-\alpha/2}w, z \big\rangle_{(0,T)} =
    \big\langle\D_{0+}^{\gamma}g, z \big\rangle_{(0,T)}
    \end{equation}
    for all $ z \in {}_0H^{\alpha/4}(0,T) $.
    Since using \cref{lem:regu-frac-deriv} yields
    \begin{equation}
    \label{eq:ineq-g}
    \nm{\D_{0+}^{\gamma}g}_{L^2(0,T)}\leqslant
    C_\gamma \nm{g}_{{}_0H^{\gamma}(0,T)},
    \end{equation}
    by \cref{lem:regu-frac-int-leq,lem:regu-frac-deriv,lem:coer} and the Lax-Milgram theorem, we claim that problem
    \cref{eq:weak-w} admits a unique solution $ w \in {}_0H^{\alpha/4}(0,T) $. In addition, inserting $z=w$ into \cref{eq:weak-w} gives
    \begin{equation}
    \label{eq:shit}
    \nm{w}_{{}_0H^{\alpha/4}(0,T)}^2+
    \lambda \nm{w}_{{}_0H^{-\alpha/4}(0,T)}^2
    \leqslant C_\alpha
    \nm{\D_{0+}^{\gamma}g}_{L^2(0,T)}
    \nm{w}_{L^2(0,T)}.
    \end{equation}
    Since \cref{lem:inter_space} and \cite[Corollary 1.7]{Lunardi2018} imply
    \[
    \nm{w}_{L^2(0,T)}\leqslant C_\alpha
    \nm{w}_{{}_0H^{\alpha/4}(0,T)}^{1/2}
    \nm{w}_{{}_0H^{-\alpha/4}(0,T)}^{1/2},
    \]
    a simple calculation gives, by \cref{eq:shit}, that
    \begin{equation}
    \label{eq:shit-10}
    \lambda^{1/4}\nm{w}_{{}_0H^{\alpha/4}(0,T)} +
    \lambda^{3/4} \nm{w}_{{}_0H^{-\alpha/4}(0,T)} \leqslant
    C_\alpha
     \nm{\D_{0+}^\gamma g}_{L^2(0,T)}.
    \end{equation}
    Observe that by \cref{eq:weak-w}
    \begin{equation} \label{eq:ode-w}
    \D_{0+}^{\alpha/2} w + \lambda \D_{0+}^{-\alpha/2} w =
    \D_{0+}^{\gamma} g \quad \text{ in } L^{2}(0,T),
    \end{equation}
    and then multiplying $\D_{0+}^{\alpha/2}w$ on both sides of the above equality and integrating on $(0,T)$ yields that
    \[
        \begin{split}
    {}&\big\|\D_{0+}^{\alpha/2} w   \big\|_{L^2(0,T)}^2 =
    \big\langle\D_{0+}^\gamma g, \D_{0+}^{\alpha/2} w   \big\rangle_{(0,T)} -
    \lambda     \big\langle\D_{0+}^{\alpha/2} w, \D_{0+}^{-\alpha/2} w  \big\rangle_{(0,T)} \\
    \leqslant &
    \nm{\D_{0+}^\gamma g}_{L^2(0,T)}
    \big\|\D_{0+}^{\alpha/2} w\big\|_{L^2(0,T)} +
    C_\alpha \lambda \nm{w}_{{}_0H^{-\alpha/4}(0,T)}
    \nm{w}_{{}_0H^{\alpha/4}(0,T)},
    \end{split}
    \]
    by \cref{lem:dual-ID}.
    Thus using \cref{eq:shit-10} and Young's inequality with $\epsilon$ yields
    \[
    \big\|\D_{0+}^{\alpha/2} w  \big\|_{L^2(0,T)}
    \leqslant C_\alpha \nm{\D_{0+}^\gamma g}_{L^2(0,T)},
    \]
    and it follows from \cref{eq:ode-w} that
    \[
    \lambda \big\|\D_{0+}^{-\alpha/2} w\big\|_{L^2(0,T)}  =
    \big\|\D_{0+}^{\alpha/2} w\,-\,\D_{0+}^{\gamma}g\big\|_{L^2(0,T)} \\
    \leqslant C_\alpha \nm{\D_{0+}^{\gamma}g}_{L^2(0,T)}.
    \]
    Hence, by \cref{lem:regu-extra-D,lem:regu-frac-int-general}, combining
    \cref{eq:ineq-g} and the above two estimates yields
    \begin{equation}\label{eq:regu-w}
    \big\|w\big\|_{{}_0H^{ \alpha/2}(0,T)}+
    \lambda \nm{w}_{{}_0H^{-\alpha/2}(0,T)}
    \leqslant C_{\alpha,\gamma} \nm{g}_{{}_0H^{\gamma}(0,T)}.
    \end{equation}

    Secondly, let us prove \cref{eq:ode-strong}. By the facts $ w \in
    {}_0H^{\alpha/2}(0,T) $ and \cref{lem:DD}, applying
    $ \D_{0+}^{-\alpha/2-\gamma} $ on both sides of \cref{eq:ode-w} yields
    \[
    \big(       \D_{0+}^{\alpha/2}+\lambda\D_{0+}^{-\alpha/2}\big)\D_{0+}^{-\alpha/2-\gamma}w = \D_{0+}^{-\alpha/2}g.
    \]
Therefore, $ y := \D^{-\alpha/2-\gamma}_{0+} w $ is the solution to problem
    \cref{eq:frac_ode_weak}, and using \cref{lem:regu-frac-int-leq,lem:regu-frac-int-general,eq:regu-w}
    proves \cref{eq:regu-ode}.
    Hence by the fact $ y \in {}_0H^{\alpha+\gamma}(0,T) $, the relation $w= \D^{\alpha/2+\gamma}_{0+} y$, \cref{lem:DD} and \cref{eq:ode-w}, we obtain that
    \[
    \D_{0+}^{\alpha+\gamma} y + \lambda \D_{0+}^{\gamma} y =    \D_{0+}^{\gamma} g,
    \]
    which proves \cref{eq:ode-strong}.

    Finally, it remains to prove \cref{eq:regu-ode-C,eq:regu-ode-C-limit-case}.
    If $1-\alpha\leqslant\gamma<1/2$, then
    by \cref{lem:Nirenberg,eq:regu-ode},
    \begin{align*}
        & \lambda^{1+(\gamma-1/2)/\alpha} \nm{y}_{C[0,T]} \\
        \leqslant{} &
        C_{\alpha,\gamma,T}
        \nm{y}_{{}_0H^{\alpha+\gamma}(0,T)}^{(1/2-\gamma)/\alpha}
        \big(
        \lambda \nm{y}_{{}_0H^\gamma(0,T)}
        \big)^{1+(\gamma-1/2)/\alpha} \\
        \leqslant{} &
        C_{\alpha,\gamma,T}\big(
        \nm{y}_{{}_0H^{\alpha+\gamma}(0,T)} +
        \lambda \nm{y}_{{}_0H^\gamma(0,T)}
        \big)\\
        \leqslant {}& C_{\alpha,\gamma,T}
        \nm{g}_{{}_0H^{\gamma}(0,T)},
    \end{align*}
    which proves \cref{eq:regu-ode-C}. If $\gamma=1/2$, then using \cref{lem:trace,eq:regu-ode} gives
            \begin{align*}
    \lambda^{1-\epsilon/2} \nm{y}_{C[0,T]} \leqslant {}&
    \frac{C_{\alpha,T}}{\epsilon}
        \big(
    \lambda \nm{y}_{{}_0H^{1/2}(0,T)}
    \big)^{1-\epsilon/2}
        \nm{y}_{{}_0H^{\alpha+1/2}(0,T)}^{ \epsilon/2}\\
    \leqslant {}&
    \frac{C_{\alpha,T}}{\epsilon}\nm{g}_{{}_0H^{1/2}(0,T)},
    \end{align*}
    for all $0<\epsilon\leqslant2/(2\alpha+1)$. This proves \cref{eq:regu-ode-C-limit-case} and thus completes the proof of this lemma.
\end{proof}
\begin{lem}
    \label{lem:basic_weak_sol}
    If $u_0=u_1=0$ and
    \[
     f \in  {}_0H^{-3\alpha/4}(0,T;L^{2}(\Omega))
     \cup
     {}_0H^{-\alpha/4}(0,T;\dot H^{-1}(\Omega)),
    \]
    then the solution to problem \cref{eq:weak_sol} is given by that
    \begin{equation}
    \label{eq:y}
    u(t) = \sum_{n=0}^{\infty}y_n(t)\phi_n, \quad 0 \leqslant t \leqslant T,
    \end{equation}
    where $ y_n $ satisfies
    \begin{equation}\small
    \label{eq:frac-ode}
    \begin{aligned}
    {}&\big\langle\D_{0+}^{\alpha/2 }y_n,z          \big\rangle_{ {}^0\!H^{\alpha/4}(0,T)} \!+\!
    \lambda_n           \big\langle\D_{0+}^{-\alpha/2}y_n,z\big\rangle_{(0,T)}
    \!=\!
    \Big\langle \D_{0+}^{-\alpha/2}\big\langle f,\phi_n\big\rangle _{\dot H^1(\Omega)},
    z\Big\rangle_{(0,T)}
    \end{aligned}
    \end{equation}
    for all $ z\in {}_0H^{\alpha/4}(0,T) $.
\end{lem}
\begin{proof}
    Let us first consider the case that $    f \in {}_0H^{-3\alpha/4}(0,T;L^{2}(\Omega)) $.
    Similar to problem \cref{eq:frac_ode_weak}, problem \cref{eq:frac-ode} admits a
    unique solution $ y_n \in {}_0H^{\alpha/4}(0,T) $, and
    \[
    \nm{y_n}_{{}_0H^{\alpha/4}(0,T)} +
    \lambda_n\|y_n\|_{{}_0H^{-\alpha/4}(0,T)}
    \leqslant C_{\alpha}
    \big\|\dual{f,\phi_n}_{\dot H^1(\Omega)}\big\|_{ {}_0H^{-3\alpha/4}(0,T)}.
    \]
    Since
    \[
\begin{split}
    \sum_{n=0}^\infty
\big\|\dual{f,\phi_n}_{\dot H^1(\Omega)}
\big\|_{ {}_0H^{-3\alpha/4}(0,T)}^2={}&
    \sum_{n=0}^\infty
\big\|\dual{f,\phi_n}_{L^2(\Omega)}
\big\|_{ {}_0H^{-3\alpha/4}(0,T)}^2\\
= {}&\nm{f}_{{}_0H^{-3\alpha/4}(0,T;L^{2}(\Omega))}^2,
\end{split}
    \]
    it follows that
    \[
\begin{split}
\sum_{n=0}^\infty \left(
\nm{y_n}_{{}_0H^{\alpha/4}(0,T)}^2+
\lambda_n\|y_n\|_{{}_0H^{-\alpha/4}(0,T)}^2
\right)
\leqslant {}&C_{\alpha}
\nm{f}_{{}_0H^{-3\alpha/4}(0,T;L^{2}(\Omega))}^2.
\end{split}
    \]
    Therefore, $ u $ defined by \cref{eq:y} belongs to $W$ and satisfies that
    \[
    \nm{u}_W
    \leqslant C_{\alpha}
 \nm{f}_{{}_0H^{-3\alpha/4}(0,T;L^{2}(\Omega))}.
    \]

    Next, let us verify that $ u $ is the solution to problem \cref{eq:weak_sol}. For
    any $ v \in W $, there exists a unique decomposition $ v = \sum_{n=0}^\infty v_n
    \phi_n $, and
    \[
    \sum_{n=0}^\infty \Big(
    \lambda_n\|v_n\|_{{}_0H^{-\alpha/4}(0,T)}^2
    +   \nm{v_n}_{{}_0H^{\alpha/4}(0,T)}^2
    \Big) = \nm{v}_W^2.
    \]
    It is evident that
    \begin{align*}
        \big\langle \D_{0+}^{\alpha/2} u, v \big\rangle _{H^{\alpha/4}(0,T;L^2(\Omega))} & =
    \sum_{n=0}^\infty  \big\langle \D_{0+}^{\alpha/2} y_n, v_n \big\rangle _{H^{\alpha/4}(0,T)},\\
    \big\langle
    \nabla \D_{0+}^{-\alpha/4} u, \nabla \D_{T-}^{-\alpha/4} v
    \big\rangle_{ \Omega\times(0,T)} & =
    \sum_{n=0}^\infty \lambda_n  \big\langle \D_{0+}^{-\alpha/2}y_n, v_n \big\rangle _{(0,T)}.
    \end{align*}
    Since $ f \in {}_0H^{-3\alpha/4}(0,T;L^{2}(\Omega)) \subset \widehat W^* $, we also have
    \begin{align*}
    \big\langle f, \D_{T-}^{-\alpha/2} v \big\rangle _{\widehat W}
    ={}&
    \sum_{n=0}^\infty   \Big\langle \D_{0+}^{-\alpha/2}\big\langle f,\phi_n\big\rangle_{\dot H^1(\Omega)},
    v_n\Big\rangle_{(0,T)}.
    \end{align*}
    Combining the above three equations and \cref{eq:frac-ode} proves that $ u $ given
    by \cref{eq:y} fulfills \cref{eq:weak_sol} for all $ v \in W $, and therefore it
    is indeed the solution to problem \cref{eq:weak_sol}. Since the proof of the case that $     f \in {}_0H^{-\alpha/4}(0,T;\dot H^{-1}(\Omega)) $ is similar. This completes the proof of
    the lemma.
\end{proof}
Combining \cref{lem:regu-ode,lem:basic_weak_sol},
we readily conclude the following regularity results for the weak solution $ u $ to
problem \cref{eq:model}.
\begin{thm}
  \label{thm:regu-pde}
  Assume that $\beta\geqslant0,\gamma\geqslant-3\alpha/4$ or $\beta\geqslant-1,\gamma\geqslant-\alpha/4$. If $u_0=u_1=0$ and $ f \in {}_0H^\gamma(0,T;\dot H^\beta(\Omega)) $, then
    \begin{equation}
    \label{eq:regu_pde_strong}
    \D_{0+}^{\alpha+\gamma} u - \Delta \D_{0+}^{\gamma} u
    =  \D_{0+}^{\gamma}f\quad\text{in~~}L^2(0,T;\dot H^{\beta}(\Omega)),
  \end{equation}
  and
  \[
    \nm{u}_{{}_0H^{\alpha+\gamma}(0,T;\dot H^\beta(\Omega))} +
    \nm{u}_{{}_0H^\gamma(0,T;\dot H^{2+\beta}(\Omega))}
    \leqslant C_{\alpha,\gamma}
    \nm{f}_{{}_0H^\gamma(0,T;\dot H^\beta(\Omega))}.
  \]
  Moreover, if $ 1-\alpha \leqslant \gamma <1/2 $ then
    \begin{equation*}
    \nm{u}_{C([0,T];\dot H^{2+\beta+(2\gamma-1)/\alpha}(\Omega))}
    \leqslant C_{\alpha,\gamma,T}
    \nm{f}_{{}_0H^\gamma(0,T;\dot H^\beta(\Omega))},
  \end{equation*}
  and if $ \gamma =1/2 $ then
  \begin{equation*}
  \nm{u}_{C([0,T];\dot H^{2+\beta-\epsilon}(\Omega))}
  \leqslant     \frac{C_{\alpha,T}}{\epsilon}
  \nm{f}_{{}_0H^{1/2}(0,T;\dot H^\beta(\Omega))},
  \end{equation*}
  for all $0<\epsilon\leqslant2/(2\alpha+1)$.
\end{thm}
\begin{rem}\label{rem:C0T}
    By \cref{lem:regu-ode,lem:basic_weak_sol}, if $
    1-\alpha\leqslant \gamma <1/2 $, then the series in \cref{eq:y} converges to $ u $
    in $ \dot H^{2+\beta+(2\gamma-1)/\alpha}(\Omega) $ uniformly on $ [0,T] $, so that
    $ u\in C([0,T]; \dot H^{2+\beta+(2\gamma-1)/\alpha}(\Omega)) $. Similarly, if
    $\gamma=1/2$, then $ u\in C([0,T]; \dot H^{2+\beta-\epsilon}(\Omega) ) $ for all
    $0<\epsilon\leqslant2/(2\alpha+1)$.
\end{rem}

\begin{rem}
  We observe that \cite[Theorem 4.21]{Bazhlekova2001} has already contained the
  regularity estimate
  \[
  \begin{split}
{}&  \nm{u}_{L^p(0,T;L^q(\Omega))} +
  \nm{\D_{0+}^\alpha u}_{L^p(0,T;L^q(\Omega))} +
  \nm{\Delta u}_{L^p(0,T;L^q(\Omega))} \\
  \leqslant{}&
  C_{\alpha,p,q}\nm{f}_{L^p(0,T;L^q(\Omega))},
  \end{split}
  \]
  where $ 1 < p,q < \infty $. In the case $ p =q= 2 $, this
  result implies
  \[
    \nm{u}_{{}_0H^\alpha(0,T;L^2(\Omega))} + \nm{u}_{L^2(0,T;\dot H^2(\Omega))}
    \leqslant C_{\alpha} \nm{f}_{L^2(0,T;L^2(\Omega))}.
  \]
\end{rem}
\begin{rem}
    Let us introduce a simple example to demonstrate that a smooth source term $ f $
    with $f(0)\neq0$ can not guarantee a smooth solution. For any $ \beta > 0 $, define the Mittag-Leffler function $
    E_{\alpha,\beta}(z) $ by
    \[
    E_{\alpha,\beta}(z) := \sum_{n=0}^\infty
    \frac{z^n}{\Gamma(n\alpha+\beta)},
    \quad z \in \mathbb C,
    \]
    and this function admits a growth estimate that \cite{Podlubny1998}
    \begin{equation}
    \label{eq:ml_growth}
    \snm{E_{\alpha,\beta}(-t)} \leqslant
    \frac{C_{\alpha,\beta}}{1+t},
    \quad t > 0.
    \end{equation}
    For any $ t \geqslant 0 $, let
    \begin{align*}
    w(t) &:= t^\alpha/\Gamma(\alpha+1), \\
    y(t) &:= -\lambda^{-1} E_{\alpha,1}(-\lambda t^\alpha) + \lambda^{-1},
    \end{align*}
    where $ \lambda $ is a positive constant, then a straightforward calculation yields
    \[
    \D_{0+}^\alpha y+\lambda y = 1, \quad t \geqslant 0.
    \]
    If $1<\alpha<3/2$, then by \cref{eq:ml_growth} we obtain
    \begin{align*}
    & \lambda^{2-3/\alpha} \nm{(y-w)''}_{L^2(0,T)}^2 \\
    ={} &
    \int_0^T \lambda^{4-3/\alpha} t^{4\alpha-4}
    E^2_{\alpha,2\alpha-1}(-\lambda t^\alpha) \, \mathrm{d}t \\
    \leqslant{} &
    C_\alpha \int_0^T \lambda^{4-3/\alpha} t^{4\alpha-4}
    (1+\lambda t^\alpha)^{-2} \, \mathrm{d}t \\
    \leqslant{} &
    C_\alpha \bigg(
    \int_0^{\lambda^{-1/\alpha}} \lambda^{4-3/\alpha} t^{4\alpha-4}
    \, \mathrm{d}t +
    \int_{\lambda^{-1/\alpha}}^\infty \lambda^{2-3/\alpha} t^{2\alpha-4}
    \, \mathrm{d}t
    \bigg) \\
    \leqslant{} &
    C_\alpha.
    \end{align*}
    If $3/2\leqslant\alpha<2$, then an analogous deduction gives
        \begin{align*}
    \lambda^{2-5/\alpha} \nm{(y-w)'''}_{L^2(0,T)}^2
    \leqslant{} &
    C_\alpha.
    \end{align*}
    Now, assume that $ u_0 = u_1 = 0 $. If $1<\alpha<3/2$ and $ f(t) = v \in \dot
    H^{3/\alpha-2}(\Omega) $, $ 0 \leqslant t \leqslant T $, then by the techniques used in
    the proof of \cref{thm:regu-pde} we obtain
    \[
    \nm{u - t^\alpha/\Gamma(\alpha+1) v}_{H^2(0,T;L^2(\Omega))}
    \leqslant C_\alpha \nm{v}_{\dot H^{3/\alpha-2}(\Omega)}.
    \]
    Similarly, if $3/2\leqslant\alpha<2$ and $ f(t) = v \in \dot
    H^{5/\alpha-2}(\Omega) $, $ 0 \leqslant t \leqslant T $, then
    \[
    \nm{u - t^\alpha/\Gamma(\alpha+1) v}_{H^3(0,T;L^2(\Omega))}
    \leqslant C_\alpha \nm{v}_{\dot H^{5/\alpha-2}(\Omega)}.
    \]
    Consequently, the temporal regularity of $ u $ is essentially determined by $
    t^\alpha v $, and its temporal regularity can not exceed $ H^{\alpha+1/2} $.
\end{rem}

Analogously to \cref{thm:regu-pde}, we have the following theorem.
\begin{thm}
    \label{thm:regu-pde-dual}
    Assume that $\beta\geqslant0,\gamma\geqslant-3\alpha/4$ or $\beta\geqslant-1,\gamma\geqslant-\alpha/4$.
    If $ q \in {}^0\!H^\gamma(0,T;\dot H^\beta(\Omega)) $, then there exists a unique
    \[
    w\in {}^0\!H^{\alpha+\gamma}(0,T;\dot H^\beta(\Omega))\cap
    {}^0\!H^\gamma(0,T;\dot H^{2+\beta}(\Omega))
    \]
    such that
    \begin{equation*}
    \D_{T-}^{\alpha+\gamma} w - \Delta \D_{T-}^{\gamma} w
    =  \D_{T-}^{\gamma}q\quad\text{in~~}L^2(0,T;\dot H^{\beta}(\Omega)),
    \end{equation*}
    and
    \[
    \nm{w}_{{}^0\!H^{\alpha+\gamma}(0,T;\dot H^\beta(\Omega))} +
    \nm{w}_{{}^0\!H^\gamma(0,T;\dot H^{2+\beta}(\Omega))}
    \leqslant C_{\alpha,\gamma}
    \nm{q}_{{}^0\!H^\gamma(0,T;\dot H^\beta(\Omega))}.
    \]
\end{thm}
\subsection{Transposition method}
In this subsection, we use the transposition method \cite{Lions-I} to investigate the regularity of
problem \cref{eq:model} with more general data.
Let
\[
G := {{}^0\!H}^\alpha(0,T;L^{2}(\Omega))
\cap L^{2}(0,T;\dot H^{2}(\Omega)),
\]
and endow this space with the norm
\[
\nm{\cdot}_G := \max\left\{
\nm{\cdot}_{{}^0\!H^\alpha(0,T;L^2(\Omega))},\
\nm{\cdot}_{L^2(0,T;\dot H^2(\Omega))}
\right\}.
\]
In addition, define
\[
G_0  := \left\{ \left( \D_{T-}^{\alpha-1} v\right)(0): \ v \in G \right\},
\quad
G_1 := \left\{ \left( \D_{T-}^{\alpha-2} v \right)(0):\ v \in G \right\},
\]
and we equip them respectively with the norms
\[
\nm{v_0}_{G_0} :=
\inf_{
    \substack{
        v \in G \\
        ( \D_{T-}^{\alpha-1} v)(0) = v_0
    }
} \nm{v}_G, \quad
\nm{v_1}_{G_1} :=
\inf_{
    \substack{
        v \in G \\
        ( \D_{T-}^{\alpha-2} v )(0) = v_1
    }
} \nm{v}_G.
\]

Assuming that $ f \in G^* $, $ u_0 \in G_0^* $ and $ u_1 \in G_1^* $, we call $ u \in
L^2(0,T;L^2(\Omega)) $ a weak solution to problem \cref{eq:model} if
\begin{equation}
\label{eq:weak_sol_general}
\begin{aligned}
& \dual{u, \D_{T-}^\alpha v - \Delta v}_{\Omega \times (0,T)} \\
={} &
\dual{f,v}_G + \dual{u_0,\left( \D_{T-}^{\alpha-1}v \right)(0)}_{G_0} +
\dual{u_1, \left( \D_{T-}^{\alpha-2} v \right)(0)}_{G_1}
\end{aligned}
\end{equation}
for all $ v \in G $.

\begin{thm}
    Problem \cref{eq:weak_sol_general} admits a unique solution $ u $, and
    \begin{equation}
    \label{eq:weak-general-regu}
    \nm{u}_{L^2(0,T;L^2(\Omega))}\leqslant C_{\alpha}
    \Big( \nm{f}_{G^*}+ \nm{u_0}_{G_0^*}+ \nm{u_1}_{G_1^*} \Big).
    \end{equation}
\end{thm}
\begin{proof}
    For any $ v\in L^2(0,T;L^2(\Omega)) $, \cref{thm:regu-pde-dual} implies that there
    exists a unique $ w\in G $ such that $ \D_{T-}^{\alpha}w-\Delta w = v $ and
    \[
    \nm{w}_G \leqslant C_\alpha \nm{v}_{L^2(0,T;L^2(\Omega))}.
    \]
    Moreover, for any $ w\in G $, \cref{lem:regu-frac-deriv} implies that
    \[
    \nm{\D_{T-}^{\alpha}w}_{L^2(0,T;L^2(\Omega))} \leqslant
    C_\alpha
    \nm{w}_{{}^0\!H^\alpha(0,T;L^2(\Omega))}.
    \]
    Therefore, applying the Babu\v{s}ka-Lax-Milgram theorem \cite{Babuska1971} yields
    that there exists a unique $u\in L^2(0,T;L^2(\Omega))$ such that
    \cref{eq:weak_sol_general,eq:weak-general-regu} hold. This completes the proof of this theorem.
\end{proof}
\begin{rem}
    If $u_0\in \dot H^{-1/2}(\Omega)$ and $u_1\in \dot H^{-3/2}(\Omega)$, then an evident computation implies that $u_0\in G_0^*$ and $u_1\in G_1^*$.
    Hence, the weak solution is well-defined by \cref{eq:weak_sol_general} for $u_0\in \dot H^{-1/2}(\Omega),u_1\in \dot H^{-3/2}(\Omega)$ and $ f\in G^*$. Moreover, if $ f\in \widehat W^*$ and $ \D_{0+}^{\alpha/2}(u_0+tu_1)\in W^*$, then the weak solutions defined by \cref{eq:weak_sol,eq:weak_sol_general}
    are essentially the same.
\end{rem}
\section{A Petrov-Galerkin Method}
\label{sec:fem}
Given $ J \in \mathbb N_{>0} $, let $ 0 = t_0 < t_1 < \cdots < t_J = T $ be a
partition of $ [0,T] $. Set $\tau_j:=t_j-t_{j-1}$ and $ I_j := (t_{j-1},t_j) $ for
each $ 1 \leqslant j \leqslant J $, and define $\tau:=\max_{1\leqslant j\leqslant
J}\tau_j$. Let $ \mathcal K_h $ be a shape-regular triangulation of $\Omega $
consisting of $ d $-simplexes, and we use $ h $ to denote the maximum diameter of the
elements in $ \mathcal K_h $. Define
\begin{align*}
  S_h := {}&\left\{
    v_h \in H_0^1(\Omega):\
    v_h|_K \in P_1(K),\ \forall \,K \in \mathcal K_h
  \right\},\\
  \widehat W_\tau:={}&\{\widehat w_\tau\in L^2(0,T):\widehat w_\tau|_{I_j}\in P_0(I_j),
  \,\forall\,1\leqslant j\leqslant J\},\\
    W_\tau:={}&\{w_\tau\in H^1(0,T):\,w_\tau|_{I_j}\in P_1(I_j),
  \,\forall\,1\leqslant j\leqslant J\}.
\end{align*}
Above and throughout, $P_k(\mathcal O)(k=0,1)$ denotes the set of polynomials defined on $ \mathcal O $ with degree $ \leqslant k $, where $\mathcal O$ is either an interval or an element of $\mathcal K_h$. Moreover, define
\begin{align*}
  \widehat W_\tau\otimes S_h:={}&{\rm span}\{\widehat w_\tau v_h:\widehat w_\tau\in   \widehat W_\tau,~
  v_h\in S_h\},\\
  W_\tau \otimes S_h:={}&{\rm span}\{w_\tau v_h:w_\tau\in W_\tau,~
  v_h\in S_h\}.
\end{align*}

Assuming that $ u_0, u_1 \in S_h^* $ and $ f \in (  \widehat W_\tau\otimes S_h)^* $, we define
an approximation $ U \in W_\tau\otimes S_h $ to problem \cref{eq:model} by that $
U(0) = u_{0,h} $ and
\begin{equation}
\label{eq:algor_sol}
\begin{aligned}
& \dual{\D_{0+}^{\alpha-1}U',V}_{\Omega \times (0,T)} +
\dual{\nabla U,\nabla V}_{\Omega \times (0,T)} \\
={} &
\dual{f,V}_{  \widehat W_\tau\otimes S_h} +
\big\langle u_1,\dual{\D_{0+}^{\alpha}t,V}_{(0,T)}
\big\rangle _{S_h}
\end{aligned}
\end{equation}
for all $ V\in   \widehat W_\tau\otimes S_h$, where $ u_{0,h}$ is the Ritz
projection of $ u_0 $ onto $ S_h $ if $ u_0 \in H_0^{1}(\Omega) $ and $ u_{0,h}$ is
the $ L^2(\Omega) $-orthogonal projection of $ u_0 $ onto $ S_h $ if $ u_0 \in S_h^*
\setminus H_0^{1}(\Omega) $. Here the Ritz projection $R_h:H_0^1(\Omega)\to S_h$ is
defined by that
\[
\dual{\nabla (v-R_hv),\nabla v_h}_{\Omega} = 0
\qquad\forall\, v_h \in S_{h},
\]
for all $v\in H_0^1(\Omega)$.

Similar to \cite[Theorem 4.1]{Li2018}, we have the following stability result.
\begin{thm}
    \label{thm:stab}
    Problem \cref{eq:algor_sol} admits a unique solution $ U $.
    Moreover, if
    $ u_0 \in \dot H^1(\Omega) $, $ u_1 \in L^2(\Omega) $ and $ f\in {}_0H^{(1-\alpha)/2}(0,T;L^2(\Omega)) $,
    then
    \begin{equation*}
    \label{eq:stab}
    \begin{aligned}
    &{} \nm{U'}_{{}_0H^{{(\alpha-1)/2}}(0,T;L^2(\Omega))} +
    \nm{U}_{C([0,T];\dot H^1(\Omega))}
    \\
    \leqslant {}&C_{\alpha,T}\big(
    \nm{u_0}_{\dot H^1(\Omega)} +
    \nm{u_1}_{L^2(\Omega)} +
        \nm{f}_{{}_0H^{(1-\alpha)/2}(0,T;L^2(\Omega))}
    \big).
    \end{aligned}
    \end{equation*}
\end{thm}
\begin{rem}
    By \cref{eq:weak_sol}, it is natural to develop another numerical algorithm for
    problem \cref{eq:model} as follows: given $u_0,u_1\in S_h^*$ and $  f \in \widehat W^*$, seek $ U \in \widehat{W}_\tau \otimes S_h $ such that
    \begin{equation}\label{eq:alter}
    \begin{aligned}
    & \big\langle \D_{0+}^{\alpha/2} U, V \big\rangle _{\Omega \times (0,T)} +
    \big\langle \nabla \D_{0+}^{-\alpha/2}U, \nabla V
    \big\rangle _{ \Omega\times(0,T)} \\
    ={}&
    \big\langle f, \D_{T-}^{-\alpha/2}V\big\rangle_{\widehat W}+
    \big\langle u_0,\langle\D_{0+}^{\alpha/2}1,V\rangle_{(0,T)}
    \big\rangle _{S_h}+
    \big\langle u_1,\langle\D_{0+}^{\alpha/2}t,V\rangle_{(0,T)}
    \big\rangle _{S_h}
    \end{aligned}
    \end{equation}
    for all $ V \in \widehat{W}_\tau \otimes S_h $.
    Analogous to \cref{thm:basic_weak_sol}, if $    \D_{0+}^{\alpha/2}(u_0+tu_1)\in W^*$, then
    \begin{align*}
{}&  \nm{U}_{{}_0H^{\alpha/4}(0,T;L^2(\Omega)) } +
\nm{U}_{{}_0H^{-\alpha/4}(0,T;\dot H^1(\Omega))}\\
    \leqslant {}&C_\alpha
    \big(
    \nm{f}_{\widehat W^*}+
        \big\|\D_{0+}^{\alpha/2}(u_0+tu_1)\big\|_{W^*}
    \big).
    \end{align*}
    Hence, by \cref{thm:stab} this algorithm is more robust than algorithm
    \cref{eq:algor_sol}; however, the computational cost of this algorithm is larger than that of the latter. To the author's knowledge, this algorithm has never been proposed before. We will pay more attention to this algorithm in future works.
\end{rem}
\subsection{Convergence analysis}
This subsection considers the convergence analysis of the Petrov-Galerkin method with $u_0=u_1=0$. Let
\[
 \sigma:=
\max_{1 \leqslant i,j \leqslant J}
\{\tau_i / \tau_{j}\}\quad{\rm and}\quad
 \rho:=
\max_{1 \leqslant j \leqslant J}
\max_{
    \substack{1 \leqslant i \leqslant J\\ \snm{i-j} \leqslant1 }
} \{\tau_i / \tau_{j}\}.
\]
 In what follows, $ a\lesssim b $ means that there exists a positive constant $C$,
depending only on $ \alpha,\rho,T,\Omega $ and the shape-regular parameter of
$\mathcal K_h$, such that $ a\leqslant Cb $. In addition, $ a\sim b $ means that $
a\lesssim b\lesssim a $.
The main results are the
following two theorems.
\begin{thm}
    \label{thm:conv-f-Hs}
    If $ f \in {}_0H^{2-\alpha}(0,T;L^2(\Omega)) $, then
    \begin{align}\small
    \nm{u\!-\!U}_{ C([0,T];\dot H^1(\Omega))}\!\!
    \lesssim &
    \big(\tau^{(3-\alpha)/2}\! \!+\!\varepsilon_1(\alpha,\tau,h)\big)\!
    \nm{f}_{{}_0\!H^{2-\alpha}(0,T;L^2(\Omega))},
    \label{eq:conv-f-Hs-linf-H1}\\
    \nm{u'\!-\!U'}_{{}_0\!H^{(\alpha-1)/2}(0,T;L^2(\Omega))}\!\!
    \lesssim &\big( \tau^{(3-\alpha)/2} \!\!+\! \varepsilon_2(\alpha,\tau,h)\big)
    \nm{f}_{{}_0\!H^{2-\alpha}(0,T;L^2(\Omega))},
    \label{eq:conv-f-Hs-Hs-L2}
    \end{align}
    where
    \begin{equation*}
    \varepsilon_1(\alpha,\tau,h):=
    \left\{
    \begin{aligned}
            &h&&{\text if~} 1<\alpha< 3/2,\\
                &  \big(1+\snm{\log h} \big)h&&{\text if~} \alpha= 3/2,\\
    &h^{3/\alpha-1}+C_\sigma\tau^{3/2-\alpha}h &&{\text if~} 3/2<\alpha<2,
    \end{aligned}
    \right.
    \end{equation*}
    and
        \begin{equation}\label{eq:eps2}
    \varepsilon_2(\alpha,\tau,h):=
    \left\{
    \begin{aligned}
    &h^{3/\alpha-1}&&{\text if~} 1<\alpha\leqslant 3/2,\\
    &C_\sigma\tau^{3/2-\alpha}h&&{\text if~} 3/2<\alpha<2.
    \end{aligned}
    \right.
    \end{equation}
\end{thm}
\begin{thm}
  \label{thm:conv-f-L2}
  If $ f \in L^{2}(0,T;L^2(\Omega)) $, then
  \begin{equation}
    \label{eq:conv-f-L2}
    \begin{aligned}
      {}& \nm{(u-U)'}_{{}_0H^{(\alpha-1)/2}(0,T;L^2(\Omega))} +
      \nm{u-U}_{ C([0,T];\dot H^1(\Omega))}  \\
      \lesssim{}& \big( \tau^{(\alpha-1)/2} +   \varepsilon_3(\alpha,\tau,h) \big)
      \nm{f}_{L^{2}(0,T;L^2(\Omega))},
    \end{aligned}
  \end{equation}
  where
        \begin{equation}\label{eq:ep3}
  \varepsilon_3(\alpha,\tau,h):=
  \left\{
  \begin{aligned}
  &h^{1-1/\alpha}&&{\text if~} 1<\alpha\leqslant 3/2,\\
  &C_\sigma\tau^{-1/2}h&&{\text if~} 3/2<\alpha<2.
  \end{aligned}
  \right.
  \end{equation}
\end{thm}
\begin{rem}
    If $1<\alpha\leqslant3/2$, then by \cref{lem:interp,thm:regu-pde}, estimates \cref{eq:conv-f-Hs-Hs-L2,eq:conv-f-L2}
are optimal with respect to the regularity of $u$; moreover,
\cref{eq:conv-f-Hs-linf-H1} is optimal and nearly optimal with respect to the regularity of $ u $ for $ 1<\alpha <3/2 $ and $ \alpha = 3/2 $, respectively. If $3/2<\alpha<2$, then all the estimates \cref{eq:conv-f-Hs-linf-H1,eq:conv-f-Hs-Hs-L2,eq:conv-f-L2} are optimal with respect to the regularity of $u$ provided that the temporal grid is quasi-uniform and $h\leqslant C\tau^{\alpha/2}$ for some positive constant $C$. However, numerical results indicate that the requirement $h\leqslant C\tau^{\alpha/2}$ is unnecessary.
\end{rem}

To prove the above two theorems, we need several interpolation operators as follows.
Define
\[
\begin{split}
W^0_\tau:={}&\{w_\tau\in {}_0H^1(0,T):\,w_\tau|_{I_j}\in P_1(I_j),
\,\forall\,1\leqslant j\leqslant J\}.
\end{split}
\]
For $0\leqslant j\leqslant J$, set
\[
\omega_j:=
\left\{
\begin{aligned}
&(t_0,t_1),&&j = 0,\\
&(t_{J-1},t_J),&&j = J,\\
(&t_{j-1},t_{j+1}),&&1\leqslant j<J,
\end{aligned}
\right.
\]
and let $\Pi_{j}$ be the $ L^2 $-orthogonal projection operator onto $ P_1(\omega_j) $.
We introduce the Cl\'ement interpolation operator $Q_\tau:L^2(0,T)\to W^0_\tau$
by that, for all $v\in L^2(0,T)$,
\[
(Q_\tau v)(t_j) ={} (\Pi_{j}v)(t_j),\quad1\leqslant j\leqslant J.
\]
For each $ 1\leqslant j\leqslant J $, define
\[
  \widehat W_{\tau,j} :=
\left\{
\widehat w_\tau\in L^2(0,T):
\widehat w_\tau|_{I_\tau}\in P_0(I_\tau),
~\forall\, I_\tau\in \mathcal T_j
\right\},
\]
where
\begin{align*}
    \mathcal T_j :=
    \left\{
    \begin{aligned}
        &\big\{
        \omega_{2i}:\
        0 \leqslant i <j/2
        \big\} \cup \big\{
        I_i:\ j < i \leqslant J
        \big\},&&j\text{ is odd},\\
        &\big\{
        \omega_{2i-1}:
        1 \leqslant i \leqslant j/2
        \big\} \cup \big\{
        I_i:\ j < i \leqslant J
        \big\},&&j\text{ is even}.
    \end{aligned}
    \right.
\end{align*}
Let $ P_{\tau,j} $ be the $ L^2 $-orthogonal projection operator onto $   \widehat W_{\tau,j} $, and define a family of modified Cl\'ement
interpolation operators $ Q_{\tau,j}: L^2(0,T) \to W^0_\tau $ by that,
for any $ v\in L^2(0,T) $,
\[
\begin{cases}
\dual{v-Q_{\tau,j} v,1}_{\omega_i}=0 &
\text{
    if $ 1\leqslant i<j $ and
    $ \snm{i-j}$ is odd,
}\\
(Q_{\tau,j} v)(p)=(Q_{\tau} v)(p)&
\text{ if $ p $ is a node of the partition $ \mathcal T_j $. }
\end{cases}
\]
By definition, it is evident that $Q_{\tau,1}=Q_\tau$.

For the above interpolant operators and the Ritz projection operator $R_h$,
we have the following standard results
\cite{Quarteroni1994Numerical,Clement1975,Dupont1980Polynomial}, which will be used implicitly in our proofs.
If $v \in {}_0H^{\gamma}(0,T)$ with $ 1/2<\gamma \leqslant 2 $, then
\[
\begin{aligned}
\nm{\left(I-Q_{\tau}\right) v}_{C[0,T]}
\leqslant{}&C_{\rho,\gamma,T}
\tau^{\gamma-1/2}
\nm{v}_{{}_0H^{\gamma}(0,T)};
\end{aligned}
\]
if $v \in {}_0H^{\gamma}(0,T)$ with $ 0\leqslant\gamma \leqslant 2 $, then
\[
\nm{(I-Q_\tau)v\,}_{L^{2}(0,T)}
\leqslant C_{\rho,\gamma}
\tau^{\gamma} \nm{v}_{{}_0H^{\gamma}(0,T)}.
\]
If $v \in {}_0H^\gamma(0,T)$ with $ 0 \leqslant\gamma \leqslant 1 $, then
\[
\nm{(I-P_{\tau,j})v}_{L^2(0,tj)}
\leqslant{}C_{\gamma}
\tau^{\gamma} \nm{v}_{{}_0H^\gamma(0,t_j)},
\]
for all $1\leqslant j\leqslant J$. If $ v \in \dot H^r(\Omega) $ with $ 1 \leqslant r \leqslant 2 $, then
\[
\nm{(I-R_h)v}_{L^2(\Omega)} +
h \nm{(I-R_h)v}_{\dot H^1(\Omega)}
\lesssim
h^r \nm{v}_{\dot H^r(\Omega)}.
\]
Except for those well-known results, we also need to establish some nonstandard error estimates of the interpolation operator $Q_{\tau,j}$.

\begin{lem}\label{lem:PR}
    If $v\in L^2(0,T)$ and $z\in {}_0H^{\gamma}(0,T)$ with $0\leqslant\gamma\leqslant1$, then
    \begin{equation}
    \label{eq:PR}
    \dual{(I-Q_{\tau,j})v,z}_{(0,t_j)}
    \leqslant C_{\gamma}
    \tau^{\gamma}
    \nm{(I-Q_{\tau,j})v}_{L^2(0,t_j)}
    \nm{z}_{{}_0H^\gamma(0,t_j)}
    \end{equation}
    for all  $1\leqslant j \leqslant J$.
\end{lem}
\begin{proof}
    If $ j $ is even then, by the definitions of $
    P_{\tau,j} $ and $ Q_{\tau,j} $,
    \begin{align*}
    {}&\dual{(I-Q_{\tau,j})v,z}_{(0,t_j)}=
    \dual{(I-Q_{\tau,j})v,(I-P_{\tau,j})z}_{(0,t_j)}  \\
    \leqslant {}&
    \nm{(I-Q_{\tau,j})v}_{L^2(0,t_j)}
    \nm{(I-P_{\tau,j})z}_{L^2(0,t_j)}\\
    \leqslant {}&C_{\gamma}
    \tau^{\gamma}
    \nm{(I-Q_{\tau,j})v}_{L^2(0,t_j)}
    \nm{z}_{{}_0H^\gamma(0,t_j)},
    \end{align*}
    which proves \cref{eq:PR}. If  $ j $ is odd then, also by the definitions of $ P_{\tau,j} $ and $ Q_{\tau,j} $,
    \begin{align}
    {}&\dual{(I-Q_{\tau,j})v,z}_{(0,t_j)} \notag \\
    ={}&
    \dual{(I-Q_{\tau,j})v,z}_{(0,t_1)}+
    \dual{(I-Q_{\tau,j})v,(I-P_{\tau,j})z}_{(t_1,t_j)} \notag \\
    \leqslant{}&
    \nm{(I-Q_{\tau,j})v}_{L^2(0,t_j)}
    \big(
    \nm{z}_{L^2(0,t_1)}+
    \nm{(I-P_{\tau,j})z}_{L^2(t_1,t_j)}
    \big) \notag \\
    \leqslant{}&
    \nm{(I-Q_{\tau,j})v}_{L^2(0,t_j)}
    \big(
    \nm{z}_{L^2(0,t_1)}+
    C_{\gamma} \tau^\gamma
    \nm{z}_{{}_0H^\gamma(0,t_j)}
    \big). \label{eq:PR-1}
    \end{align}
    Since \cref{lem:compose-ID} implies that
    $z = \D_{0+}^{-\gamma}\D_{0+}^{\gamma}z$,
    by \cref{lem:basic-2,lem:regu-frac-deriv} we have
    \[
    \begin{split}
    {}&     \nm{z}_{L^2(0,t_1)}=
    \nm{  \D_{0+}^{-\gamma}\D_{0+}^{\gamma}z}_{L^2(0,t_1)}\\
    \leqslant {}&C_\gamma t_1^\gamma \nm{\D_{0+}^{\gamma}z}_{L^2(0,t_1)}
    \leqslant C_\gamma t_1^\gamma
    \nm{z}_{{}_0H^\gamma(0,t_1)}.
    \end{split}
    \]
    Therefore,
    combining the above inequality and \cref{eq:PR-1} proves \cref{eq:PR} in the case that $ j
    $ is odd. This completes the proof of this lemma.
\end{proof}
\begin{lem}
    \label{lem:R_tau}
    If $ v\in {}_0H^{\gamma}(0,T) $ with $ 0\leqslant\gamma\leqslant1 $, then
    \begin{equation} \label{eq:R-tau-est-L2}
        \nm{(I-Q_{\tau,j})v }_{L^{2}(0,T)}
        \leqslant C_{\rho,\gamma}
        \tau^{\gamma}
        \nm{v}_{{}_0H^{\gamma}(0,T)}
    \end{equation}
    for all $ 1 \leqslant j \leqslant J $.
\end{lem}
\begin{proof}
    If $j=1$ then \cref{eq:R-tau-est-L2} is standard, and so we assume that $2\leqslant
    j\leqslant J$. By the definition of $ Q_{\tau,j} $, $ \left(Q_{\tau,j}v\right)(0) =
    \left(Q_{\tau}v\right)(0) $, $ (Q_{\tau,j}v)|_{(t_j,T)} = (Q_{\tau} v)|_{(t_j,T)}
    $, and $ \left(Q_{\tau,j}v\right)(t_i) = \left(Q_{\tau}v\right)(t_i) $ if
    $1\leqslant i\leqslant j$ and $j-i$ is even. If $ 1\leqslant i\leqslant j-1 $ and
    $j-i$ is odd, then a straightforward calculation gives
    \begin{equation}\label{eq:L-R}
        (Q_{\tau,j}v - Q_{\tau} v)(t_i) =
        \frac{2}{\snm{\omega_i}}
        \int_{\omega_i} (v-Q_{\tau} v)(t) \, \mathrm{d}t ,
    \end{equation}
    and hence
    \begin{align*}
        \nm{Q_{\tau,j}v - Q_{\tau}  v}_{L^2(\omega_i)}
        & =
        \sqrt{\frac{\snm{\omega_i}}{3}}
        \snm{(Q_{\tau,j}v -Q_{\tau} v)(t_i)}
        \leqslant\frac2{\sqrt{3}}
        \nm{(I - Q_{\tau} ) v}_{L^2(\omega_i)},
    \end{align*}
    which implies
    \[
    \nm{Q_{\tau,j}v -Q_{\tau} v}_{L^2(0,T)} \leqslant
    \frac{2}{\sqrt{3}} \nm{(I - Q_{\tau} ) v}_{L^2(0,T)}.
    \]
    It follows that
    \[
    \begin{split}
    \nm{(I-Q_{\tau,j})v}_{L^2(0,T)} \leqslant{}&
    \frac{2 +\sqrt3}{\sqrt{3}} \nm{(I-Q_{\tau} )v}_{L^2(0,T)}
    \leqslant{}C_{\rho,\gamma}
    \tau^\gamma \nm{v}_{{}_0H^\gamma(0,T)}.
    \end{split}
    \]
    This proves \cref{eq:R-tau-est-L2} and thus concludes the proof.
\end{proof}

\begin{lem}
    \label{lem:R_tau-est-Hs}
    Assume that $ 0<\beta<1/2 $. If $ v\in {}_0H^{\gamma}(0,T) $ with $ \beta+1\leqslant
    \gamma\leqslant 2 $, then
    \begin{equation}\label{eq:R-tau-est-Hs}
    \nm{\left(v-Q_{\tau,j} v\right)'}_{{}_0H^\beta(0,T)}
    \leqslant{}C_{\rho,\beta,\gamma,T}
    \tau^{\gamma-1-\beta}
    \nm{v}_{{}_0H^{\gamma}(0,T)},
    \end{equation}
    for all $   1 \leqslant j \leqslant J$.
\end{lem}
\begin{proof}
    For simplicity, set $ g := (v-Q_{\tau,j}v)' $. Following the proof of \cite[Lemma
    4.1]{Li2018Error}, we obtain
    \[
        \nm{g}_{{}_0H^\beta(0,T)}^2 \leqslant
            C_{\beta}\nm{g}_{H^\beta(0,T)}^2
            \leqslant C_{\beta,T} \big( \mathbb I_1 + \mathbb I_2 \big),
    \]
    where
    \begin{align*}
    \mathbb I_1 &:= \sum_{i=1}^J
    \int_{I_i} \int_{I_i} \frac{\snm{g(s)-g(t)}^2}{\snm{s-t}^{1+2\beta}}
    \, \mathrm{d}s\mathrm{d}t,\\
    \mathbb I_2 &:= \sum_{i=1}^J \int_{I_i} g^2(t) \left(
    (t_i-t)^{-2\beta} + (t-t_{i-1})^{-2\beta}
    \right) \, \mathrm{d}t.
    \end{align*}
    As $  \mathbb I_1$ can be estimated by that
    \begin{align*}
    \mathbb I_1  & \leqslant
    \sum_{i=1}^J \tau_i^{2(\gamma-1-\beta)} \int_{I_i} \int_{I_i}
    \frac{\snm{g(s)-g(t)}^2}{\snm{s-t}^{1+2(\gamma-1)}} \, \mathrm{d}s\mathrm{d}t \\
    & =
    \sum_{i=1}^J \tau_i^{2(\gamma-1-\beta)} \int_{I_i} \int_{I_i}
    \frac{\snm{v'(s)-v'(t)}^2}{\snm{s-t}^{1+2(\gamma-1)}} \, \mathrm{d}s\mathrm{d}t \\
    & \leqslant
    C_{\gamma,T}
    \tau^{2(\gamma-1-\beta)} \nm{v'}_{H^{\gamma-1}(0,T)}^2\\
    & \leqslant
    C_{\gamma,T}
    \tau^{2(\gamma-1-\beta)} \nm{v}_{{}_0H^{\gamma}(0,T)}^2,
    \end{align*}
    it remains to prove
    \begin{equation}
    \label{eq:I2}
    \mathbb I_2
    \leqslant C_{\rho,\beta,\gamma,T}
    \tau^{2(\gamma-1-\beta)}
    \nm{v}_{{}_0H^{\gamma}(0,T)}^2.
    \end{equation}
    By \cite[Theorem 11.2]{Lions-I} and the fact that $
    H^{\gamma-1}(0,1) $ is continuously embedded in $ H^\beta(0,1) $, a simple calculation yields
    \begin{align*}
    {}&  \int_0^1 \snm{z(t)}^2
    \left( t^{-2\beta} + (1-t)^{-2\beta} \right)  \mathrm{d}t
    \leqslant C_\beta \nm{z}_{H^\beta(0,1)}^2
    \leqslant C_{\beta,\gamma} \nm{z}_{H^{\gamma-1}(0,1)}^2 \\
    \leqslant{} &
    C_{\beta,\gamma} \bigg(
    \nm{z}_{L^2(0,1)}^2 +
    \int_0^1 \int_0^1
    \frac{\snm{z(s)-z(t)}^2}{\snm{s-t}^{1+2(\gamma-1)}} \, \mathrm{d}s\mathrm{d}t
    \bigg),
    \end{align*}
    where $ z \in H^{\gamma-1}(0,1) $.
    Hence, a standard scaling argument gives
    \begin{align}
    \mathbb I_2 & \leqslant
    C_{\beta,\gamma} \sum_{i=1}^J
    \bigg(
    \tau_i^{-2\beta} \nm{g}_{L^2(I_i)}^2 +
    \tau_i^{2(\gamma-1-\beta)} \int_{I_i} \int_{I_i}
    \frac{\snm{v'(s)-v'(t)}^2}{\snm{s-t}^{1+2(\gamma-1)}} \, \mathrm{d}s\mathrm{d}t
    \bigg) \notag \\
    & \leqslant C_{\beta,\gamma,T}
    \bigg(
    \sum_{i=1}^J \tau_i^{-2\beta} \nm{g}_{L^2(I_i)}^2 +
    \tau^{2(\gamma-1-\beta)} \nm{v'}_{H^{\gamma-1}(0,T)}^2
    \bigg)
    \notag
    \\
    & \leqslant C_{\beta,\gamma,T}
    \bigg(
    \sum_{i=1}^J \tau_i^{-2\beta} \nm{g}_{L^2(I_i)}^2 +
    \tau^{2(\gamma-1-\beta)} \nm{v}_{{}_0H^{\gamma}(0,T)}^2
    \bigg).
    \notag
    \end{align}
    To prove \cref{eq:I2}, it suffices to prove
    \begin{equation}
    \label{eq:g-L2}
    \sum_{i=1}^J \tau_i^{-2\beta} \nm{g}_{L^2(I_i)}^2
    \leqslant C_{\rho,\beta,\gamma,T}
    \tau^{2(\gamma-1-\beta)} \nm{v}_{{}_0H^{\gamma}(0,T)}^2.
    \end{equation}
    By \cref{eq:L-R}, we have
    \[
    \nm{(Q_\tau v-Q_{\tau,j}v)'}_{L^2(\omega_i)}\leqslant
    C_\rho \snm{\omega_i}^{-1}\nm{(I-Q_\tau) v}_{L^2(\omega_i)},
    \]
    for all $ 1\leqslant i\leqslant j-1 $ such that $j-i$ is odd. Hence it follows that
    \[
    \sum_{i=1}^J \tau_i^{-2\beta} \nm{(Q_\tau v-Q_{\tau,j}v)'}_{L^2(I_i)}^2
    \leqslant C_\rho
    \sum_{i=1}^J \tau_i^{-2\beta-2} \nm{(I-Q_\tau) v}_{L^2(I_i)}^2,
    \]
    which yields the inequality
    \begin{equation*}\small
    \begin{split}
    \sum_{i=1}^J \tau_i^{-2\beta} \nm{g}_{L^2(I_i)}^2
    \!\!    \leqslant\! {}&C_\rho
    \sum_{i=1}^J
    \tau_i^{-2\beta-2}\left( \nm{(I-Q_\tau) v}_{L^2(I_i)}^2\!\!+\!
    \tau_i^{2} \nm{(v-Q_\tau v)'}_{L^2(I_i)}^2\right).
    \end{split}
    \end{equation*}
    Therefore, by the standard estimates
    \[
    \begin{split}
    \sum_{i=1}^J \tau_i^{-2\beta} \nm{(v-Q_\tau v)'}_{L^2(I_i)}^2
    \leqslant {}&C_{\rho,\beta,\gamma,T}
    \tau^{2(\gamma-1-\beta)}
    \nm{v}_{{}_0H^{\gamma}(0,T)}^2,\\
    \sum_{i=1}^J \tau_i^{-2\beta-2} \nm{(I-Q_\tau) v}_{L^2(I_i)}^2
    \leqslant {}&C_{\rho,\beta,\gamma,T}
    \tau^{2(\gamma-1-\beta)}
    \nm{v}_{{}_0H^{\gamma}(0,T)}^2,
    \end{split}
    \]
    we obtain \cref{eq:g-L2}. This completes the proof of this lemma.
\end{proof}
\begin{lem}
  \label{lem:inv}
  If $w\in W_\tau$, then
  \begin{equation}
    \label{eq:inv}
    \nm{w}_{{}_0H^{\beta+1}(0,T)}
    \leqslant {}C_{\sigma,\beta,\gamma,T}
    \tau^{\gamma-\beta-1}
    \nm{w}_{{}_0H^{\gamma}(0,T)}
  \end{equation}
  for all $0\leqslant\beta<1/2$ and $0\leqslant\gamma\leqslant\beta+1$.
\end{lem}
\begin{proof}
  If $ \beta = 0 $, then \cref{eq:inv} is trivial for $ \gamma = 1 $ and standard for
  $ \gamma = 0 $, and hence applying \cite[Lemma 22.3]{Tartar2007} yields
  \cref{eq:inv} for $ 0 < \gamma < 1 $. It remains therefore to prove \cref{eq:inv}
  for $ 0 < \beta < 1/2 $ and $ 0 \leqslant \gamma \leqslant \beta+1 $. To this end,
  we assume $ 0 < \beta < 1/2 $. For any $ 1 \leqslant \gamma \leqslant \beta+1 $,
  following the proof of \cite[Lemma 4.1]{Li2018Error}, we obtain
  \begin{small}
  \begin{align*}
    {}& \nm{w'}_{{}_0H^\beta(0,T)}^2 \leqslant
    C_{\beta,T}\nm{w'}_{H^\beta(0,T)}^2  \\
    \leqslant {}&C_{\beta,T}  \sum_{i=1}^J \int_{I_i}
    \snm{w'(t)}^2 \left(
      (t_i-t)^{-2\beta} + (t-t_{i-1})^{-2\beta}
    \right) \mathrm{d}t\\
    \leqslant {}&C_{\beta,T}
    \sum_{i=1}^J \tau_i^{2\gamma-2\beta-2}\int_{I_i}
    \snm{w'(t)}^2 \left(
      (t_i-t)^{2-2\gamma} + (t-t_{i-1})^{2-2\gamma}
    \right)\mathrm{d}t\\
    \leqslant {}&C_{\sigma,\beta,\gamma,T}
    \tau^{2\gamma-2\beta-2}
    \nm{w'}_{{}_0H^{\gamma-1}(0,T)}^2.
  \end{align*}
  \end{small}
  Hence, by the estimates
  \[
    \nm{w}_{{}_0H^{\beta+1}(0,T)} \leqslant C_\beta \nm{w'}_{{}_0H^\beta(0,T)}
    \text{ and }
    \nm{w'}_{{}_0H^{\gamma-1}(0,T)} \leqslant
    C_\gamma \nm{w}_{{}_0H^\gamma(0,T)},
  \]
  which can be easily proved by \cref{lem:regu-frac-deriv}, we have
  \begin{align*}
    \nm{w}_{{}_0H^{\beta+1}(0,T)} \leqslant {}& C_\beta
    \nm{w'}_{{}_0H^{\beta}(0,T)}
    \leqslant C_{\sigma,\beta,\gamma,T}
    \tau^{\gamma-\beta-1}
    \nm{w'}_{{}_0H^{\gamma-1}(0,T)}\\
    \leqslant {}&
    C_{\sigma,\beta,\gamma,T}
    \tau^{\gamma-\beta-1}
    \nm{w}_{{}_0H^{\gamma}(0,T)},
  \end{align*}
  namely \cref{eq:inv} holds for $ 1 \leqslant \gamma \leqslant \beta+1 $. In
  addition, for any $ 0 \leqslant \gamma < 1 $, since we have already proved
  \begin{align*}
    \nm{w}_{{}_0H^1(0,T)} &\leqslant C_{\sigma,\gamma,T}
    \tau^{\gamma-1} \nm{w}_{{}_0H^\gamma(0,T)}, \\
    \nm{w}_{{}_0H^{\beta+1}(0,T)} &\leqslant {}C_{\sigma,\beta,T}
    \tau^{-\beta}
    \nm{w}_{{}_0H^{1}(0,T)},
  \end{align*}
  it is clear that \cref{eq:inv} holds. This completes the proof.
\end{proof}

\begin{lem}
    \label{lem:conv-theta-f-Hs}
    If $ f \in {}_0H^{2-\alpha}(0,T;L^2(\Omega)) $, then
    \begin{equation}\label{eq:conv-theta-f-Hs}
    \begin{aligned}
    & \nm{(U-Q_{\tau,j} R_h u)'}_{{}_0H^{(\alpha-1)/2}(0,t_j;L^2(\Omega))} +
    \nm{(U-Q_{\tau,j}R_h u)(t_j)}_{\dot H^1(\Omega)} \\
    \lesssim{} &\big( \tau^{(3-\alpha)/2} \!\!+\! \varepsilon_2(\alpha,\tau,h)\big)
    \nm{f}_{{}_0\!H^{2-\alpha}(0,T;L^2(\Omega))},
    \end{aligned}
    \end{equation}
    for all $1\leqslant j\leqslant J$, where $\varepsilon_2(\alpha,\tau,h)$ is defined by
    \cref{eq:eps2}.
\end{lem}
\begin{proof}
    By \cref{eq:regu_pde_strong,eq:algor_sol},
    \[
    \dual{
        \D_{0+}^{\alpha-1} (u-U)', \xi_j'
    }_{\Omega\times(0,t_j)} +
    \dual{\nabla (u-U), \nabla \xi_j'}_{\Omega\times(0,t_j)} = 0,
    \]
    where $ \xi_j = U - Q_{\tau,j} R_h u $.
    As $\xi_j(0)=0$, using integration by
    parts gives
    \begin{align*}
        2\dual{ \nabla \xi_j, \nabla \xi_j'}_{ \Omega\times(0,t_j)} =
        \nm{\xi_j(t_j)}_{\dot H^1(\Omega)}^2,
    \end{align*}
    and a simple calculation then yields
    \[
\begin{split}
{}& \dual{
    \D_{0+}^{\alpha-1} \xi_j', \xi_j'
}_{\Omega\times(0,t_j)} +
\frac12 \nm{\xi_j(t_j)}_{\dot H^1(\Omega)}^2 \\
={}& \dual{ \nabla(u-Q_{\tau,j} R_hu),
    \nabla\xi_j'}_{\Omega\times(0,t_j)}+
\dual{
    \D_{0+}^{\alpha-1}(u-Q_{\tau,j} R_hu)', \xi_j'
}_{\Omega\times(0,t_j)}.
\end{split}
    \]
     It follows from \cref{lem:regu-frac-deriv,lem:coer} that
    \[
    \nm{\xi_j'}_{{}_0H^{(\alpha-1)/2}(0,t_j;L^2(\Omega))}^2 +
    \nm{\xi_j(t_j)}_{\dot H^1(\Omega)}^2 \lesssim E_1 + E_2+E_3,
    \]
    where
    \[
    \begin{split}
                E_1:={}&\dual{
        \D_{0+}^{\alpha-1}(u-Q_{\tau,j} u)', \xi_j'
    }_{\Omega\times(0,t_j)} ,\\
    E_2 := {}&\dual{ \nabla(u-Q_{\tau,j} R_hu),
                \nabla\xi_j'}_{\Omega\times(0,t_j)},\\
                        E_3:={}&\dual{
                            \D_{0+}^{\alpha-1}\big(Q_{\tau,j} (u-R_hu)\big)', \xi_j'
                        }_{\Omega\times(0,t_j)} .
    \end{split}
    \]

   Next, let us estimate $E_1,E_2$ and $E_3$ one by one.
    Since applying \cref{lem:R_tau-est-Hs} indicates
    \[\small
    \nm{(u-Q_{\tau,j} u)'}_{{}_0H^{(\alpha-1)/2}
        (0,T;L^2(\Omega))}\lesssim{}
    \tau^{(3-\alpha)/2}\nm{u}_{{}_0H^{2}(0,T;L^2(\Omega))},
    \]
    by \cref{lem:regu-frac-deriv,lem:coer,thm:regu-pde} we obtain
    \begin{align*}
 E_1 \lesssim{}& \nm{\D_{0+}^{(\alpha-1)/2}(u-Q_{\tau,j}u)'}_{L^2(0,t_j;L^2(\Omega))}\nm{\D_{0+}^{(\alpha-1)/2}\xi_j'}_{L^2(0,t_j;L^2(\Omega))}  \\ \lesssim{}&\nm{(u-Q_{\tau,j}u)'}_{{}_0H^{(\alpha-1)/2}(0,T;L^2(\Omega))}
\nm{\xi'_j}_{{}_0H^{(\alpha-1)/2}(0,t_j;L^2(\Omega))}\\
\lesssim{}&\tau^{(3-\alpha)/2}
\nm{f}_{{}_0H^{2-\alpha}(0,T;L^2(\Omega))}
\nm{\xi'_j}_{{}_0H^{(\alpha-1)/2}(0,t_j;L^2(\Omega))}.
\end{align*}
By \cref{lem:PR} and the definition of $ R_h $,
    \begin{align*}
    E_2 & = \dual{\nabla(u-Q_{\tau,j} u), \nabla\xi_j'}_{\Omega \times (0,t_j)} =
    \dual{(Q_{\tau,j} -I) \Delta u, \xi'_j}_{\Omega \times (0,t_j)} \\
    & \lesssim \tau^{(\alpha-1)/2} \nm{(I-Q_{\tau,j} )u}_{L^{2}(0,T;\dot H^2(\Omega))}
    \nm{\xi_j'}_{{}_0H^{(\alpha-1)/2}(0,t_j;L^2(\Omega))},
    \end{align*}
    so that using \cref{thm:regu-pde,lem:R_tau} gives
    \begin{align*}
    E_2& \lesssim \tau^{(3-\alpha)/2}
    \nm{f}_{{}_0H^{2-\alpha}(0,T;L^2(\Omega))}
    \nm{\xi_j'}_{{}_0H^{(\alpha-1)/2}(0,t_j;L^2(\Omega))}.
    \end{align*}
    If $1<\alpha\leqslant3/2$, then applying \cref{lem:R_tau-est-Hs} indicates
    \[\small
    \begin{split}
{}&\big\|\big(Q_{\tau,j} (u-R_hu)\big)'    \big\|_{{}_0H^{(\alpha-1)/2}(0,T;L^2(\Omega))}\\
    \lesssim    {}&
    \nm{(I-R_h)u}_{{}_0H^{(\alpha+1)/2}(0,T;L^2(\Omega))} \\
    \lesssim{}&h^{3/\alpha-1}
    \nm{u}_{{}_0H^{(\alpha+1)/2}(0,T;\dot H^{3/\alpha-1}(\Omega))},
    \end{split}
    \]
    and if $3/2<\alpha<2$ then, by \cref{lem:R_tau-est-Hs,lem:inv},
        \[\small
    \begin{split}
{}& \big\|\big(Q_{\tau,j} (u-R_hu)\big)'    \big\|_{{}_0H^{(\alpha-1)/2}(0,T;L^2(\Omega))}\\
    \lesssim{}&C_\sigma \tau^{3/2-\alpha}\big\|\big(Q_{\tau,j} (u-R_hu)\big)'    \big\|_{{}_0H^{1-\alpha/2}(0,T;L^2(\Omega))}\\
    \lesssim{}&C_\sigma \tau^{3/2-\alpha}
    \nm{(I-R_h)u}_{{}_0H^{2-\alpha/2}(0,T;L^2(\Omega))} \\
    \lesssim{}&C_\sigma     \tau^{3/2-\alpha}h
    \nm{u}_{{}_0H^{2-\alpha/2}(0,T;\dot H^{1}(\Omega))}.
    \end{split}
    \]
    Therefore, by \cref{lem:interp,lem:regu-frac-deriv,lem:coer,thm:regu-pde} we get
    \begin{align*}
 E_3 \lesssim{}& \Big\|\D_{0+}^{(\alpha-1)/2}\big(Q_{\tau,j} (u-R_hu)\big)'\Big\|_{L^2(0,t_j;L^2(\Omega))}
 \Big\|\D_{0+}^{(\alpha-1)/2}\xi_j'\Big\|_{L^2(0,t_j;L^2(\Omega))}  \\
 \lesssim{}& \big\|\big(Q_{\tau,j} (u-R_hu)\big)'   \big\|_{{}_0H^{(\alpha-1)/2}(0,T;L^2(\Omega))}
    \nm{\xi'_j}_{{}_0H^{(\alpha-1)/2}(0,t_j;L^2(\Omega))}\\
    \lesssim{}&
    \varepsilon_2(\alpha,\tau,h)
    \nm{f}_{{}_0H^{2-\alpha}(0,T;L^2(\Omega))}
    \nm{\xi'_j}_{{}_0H^{(\alpha-1)/2}(0,t_j;L^2(\Omega))}.
    \end{align*}

    Finally, combining the estimates of $ E_1,E_2$ and $ E_3 $ and the Young's inequality with $\epsilon$, we obtain that
    \[
    \begin{split}
    {}&\nm{\xi'_j}_{{}_0H^{(\alpha-1)/2}(0,t_j;L^2(\Omega))} +
    \nm{\xi_j(t_j)}_{\dot H^1(\Omega)}\\
    \lesssim{}&\big( \tau^{(3-\alpha)/2} +  \varepsilon_2(\alpha,\tau,h)  \big)
    \nm{f}_{{}_0H^{2-\alpha}(0,T;L^2(\Omega))},
    \end{split}
    \]
    for all $1\leqslant j\leqslant J$. This proves \cref{eq:conv-theta-f-Hs} and thus concludes the proof.
\end{proof}
\begin{rem}
    Assume that $f\in{}_0H^{2-\alpha}(0,T;L^2(\Omega))$ with $ 3/2 < \alpha < 2 $. By
    \cref{lem:interp,thm:regu-pde} we have
    \[
    u \in{}_0H^{(\alpha+1)/2}(0,T;\dot H^{3/\alpha-1}(\Omega))
    \cap {}_0H^{2-\alpha/2}(0,T;\dot H^1(\Omega)).
    \]
    Therefore, $ (R_hu)' = R_hu' \in {}_0H^{(\alpha-1)/2}(0,T;\dot H^{3/\alpha-1}(\Omega)) $ may not
    make sense since $ 3/\alpha-1 < 1 $, but $ (R_hu)' = R_hu' \in
    {}_0H^{1-\alpha/2}(0,T;\dot H^1(\Omega)) $ makes sense indeed. This is the reason
    why we use
    \[
    \tau^{3/2-\alpha} \nm{(Q_{\tau,j}(u-R_hu))'}_{{}_0H^{1-\alpha/2}(0,T;L^2(\Omega))}
    \]
    to bound
    \[
    \nm{(Q_{\tau,j}(u-R_hu))'}_{{}_0H^{(\alpha-1)/2}(0,T;L^2(\Omega))},
    \]
when estimating the term $E_3$ in the proof of the above lemma.
\end{rem}
Analogously to \cref{lem:conv-theta-f-Hs}, we have the following lemma.
\begin{lem}
    \label{lem:conv-theta-f-L2}
    If $ f \in L^{2}(0,T;L^2(\Omega)) $, then
        \begin{equation*}
    \begin{aligned}
    & \nm{(U-Q_{\tau,j} R_h u)'}_{{}_0H^{(\alpha-1)/2}(0,t_j;L^2(\Omega))} +
    \nm{(U-Q_{\tau,j}R_h u)(t_j)}_{\dot H^1(\Omega)} \\
    \lesssim{} &\big( \tau^{(\alpha-1)/2} +C_\sigma\tau^{-1/2}h\big)
    \nm{f}_{L^{2}(0,T;L^2(\Omega))},
    \end{aligned}
    \end{equation*}
    for all $1\leqslant j\leqslant J$.
\end{lem}

\medskip\noindent{\bf Proof of \cref{thm:conv-f-Hs}.}
By \cref{lem:interp,thm:regu-pde}, we have
\[
\small
\begin{split}
{}&\nm{(I-Q_\tau)R_hu}_{C([0,T];\dot H^1(\Omega))}
\lesssim\tau^{(3-\alpha)/2} \nm{f}_{{}_0H^{2-\alpha}(0,T;L^2(\Omega))},\\
{}&\nm{(I-R_h)u}_{C([0,T];\dot H^1(\Omega))}
\lesssim
\left\{
\begin{aligned}
&h\nm{f}_{{}_0H^{2-\alpha}(0,T;L^2(\Omega))}&&{\mathrm {if}~} 1<\alpha<3/2,\\
&
    \frac{h^{1-\epsilon}}{
\epsilon}
\nm{f}_{{}_0H^{1/2}(0,T;L^2(\Omega))}&&{\mathrm {if}~} \alpha=3/2,\\
&h^{3/\alpha-1}\nm{f}_{{}_0H^{2-\alpha}(0,T;L^2(\Omega))}&&{\mathrm {if}~} 3/2<\alpha<2,
\end{aligned}
\right.
\end{split}
\]
where $0<\epsilon\leqslant1/2$. Therefore, if $\alpha=3/2$, then letting $\epsilon = 1/(2+\snm{\log h})$ yields
\[
\nm{(I-R_h)u}_{C([0,T];\dot H^1(\Omega))}
\lesssim \big(1+\snm{\log h}\big)h\nm{f}_{{}_0H^{1/2}(0,T;L^2(\Omega))}.
\]
Since $(Q_{\tau,j}R_h u)(t_j) = (Q_{\tau}R_h u)(t_j)$ for all $1\leqslant j\leqslant J$, we obtain
\[
\small
\begin{split}
{}&\nm{Q_\tau R_hu - U}_{C([0,T];\dot H^1(\Omega))}
={}
\max_{1 \leqslant j \leqslant J}
\nm{(Q_{\tau}R_hu - U)(t_j)}_{\dot H^1(\Omega)}\\
={}&
\max_{1\leqslant j \leqslant J}
\nm{(Q_{\tau,j}R_hu - U)(t_j)}_{\dot H^1(\Omega)}\lesssim \big( \tau^{(3-\alpha)/2} +   \varepsilon_2(\alpha,\tau,h)  \big)
\nm{f}_{{}_0H^{2-\alpha}(0,T;L^2(\Omega))},
\end{split}
\]
by \cref{lem:conv-theta-f-Hs}, and hence,
\begin{align}
{}&\nm{R_hu - U}_{C([0,T];\dot H^1(\Omega))} \notag\\
\leqslant {}&
\nm{(I-Q_\tau)R_hu}_{C([0,T];\dot H^1(\Omega))} +
\nm{Q_\tau R_hu - U}_{C([0,T];\dot H^1(\Omega))} \notag\\
\lesssim {}&
\big(\tau^{(3-\alpha)/2}    +\varepsilon_2(\alpha,\tau,h)\big)\nm{f}_{{}_0H^{2-\alpha}(0,T;L^2(\Omega))}.\label{eq:Rh-U}
\end{align}
Therefore, using the triangle inequality
\begin{align*}\small
    \nm{u-U}_{C([0,T];\dot H^1(\Omega))} \leqslant
    \nm{u-R_hu}_{C([0,T];\dot H^1(\Omega))} \!\!+\!
    \nm{R_hu - U}_{C([0,T];\dot H^1(\Omega))}
\end{align*}
proves \cref{eq:conv-f-Hs-linf-H1}. Since the proof of \cref{lem:conv-theta-f-Hs} yields that
\[\small
\begin{split}
{}&\big\|\big(u-Q_{\tau,j}R_hu)\big)'    \big\|_{{}_0H^{(\alpha-1)/2}(0,T;L^2(\Omega))}\\
\leqslant{}&    \nm{(u-Q_{\tau,j} u)'}_{{}_0H^{(\alpha-1)/2}
    (0,T;L^2(\Omega))}+\big\|\big(Q_{\tau,j} (u-R_hu)\big)'    \big\|_{{}_0H^{(\alpha-1)/2}(0,T;L^2(\Omega))}\\
\lesssim{}& \big( \tau^{(3-\alpha)/2} +     \varepsilon_2(\alpha,\tau,h)  \big)
\nm{f}_{{}_0H^{2-\alpha}(0,T;L^2(\Omega))},
\end{split}
\]
 \cref{eq:conv-f-Hs-Hs-L2} follows form the above inequality and \cref{lem:conv-theta-f-Hs}. This concludes the proof of the theorem.
\hfill\ensuremath{\blacksquare}

\medskip\noindent{\bf Proof of \cref{thm:conv-f-L2}.} In view of \cref{lem:conv-theta-f-L2}, the proof of the case $3/2<\alpha<2$ is completely analogous to that of \cref{thm:conv-f-Hs}. Therefore, we only give the proof for $1<\alpha\leqslant3/2$ using the theory of interpolation space.
As \cref{thm:regu-pde,thm:stab}
imply
\[
\nm{(u-U)'}_{{}_0H^{(\alpha-1)/2}(0,T;L^2(\Omega))}\lesssim
\nm{f}_{{}_0H^{(1-\alpha)/2}(0,T;L^2(\Omega))},
\]
by \cref{eq:conv-f-Hs-Hs-L2} and the fact
\[\small
\begin{split}
L^2(0,T;L^2(\Omega)) =
\big[
{}_0H^{(1-\alpha)/2}(0,T;L^2(\Omega)),
{}_0H^{2-\alpha}(0,T;L^2(\Omega))
\big]_{(\alpha-1)/(3-\alpha),2},
\end{split}
\]
applying \cite[Lemma 22.3]{Tartar2007} yields
\[
\small
\begin{split}
\nm{(u\!-\!U)'}_{{}_0H^{(\alpha-1)/2}(0,T;L^2(\Omega))}\!\lesssim\!
\big(\tau^{(3-\alpha)/2}\!+\!h^{3/\alpha-1}\big)^{(\alpha-1)/(3-\alpha)}
\!\!\,\nm{f}_{L^{2}(0,T;L^2(\Omega))}.
\end{split}
\]
Hence, from the inequality
\[
\big(
\tau^{(3-\alpha)/2}+h^{3/\alpha-1}
\big)^{(\alpha-1)/(3-\alpha)} <
\tau^{(\alpha-1)/2}+h^{1-1/\alpha},
\]
it follows that
\begin{equation}\label{eq:conv-f-L2-Hs-L2}
    \nm{(u-U)'}_{{}_0H^{(\alpha-1)/2}(0,T;L^2(\Omega))}
    \lesssim
    \big(\tau^{(\alpha-1)/2}+h^{1-1/\alpha} \big)
    \nm{f}_{L^{2}(0,T;L^2(\Omega))}.
\end{equation}

In addition, \cref{thm:regu-pde,thm:stab} imply
\[
\nm{R_hu-U}_{C([0,T];\dot H^1(\Omega))}
\lesssim
\nm{f}_{{}_0H^{(1-\alpha)/2}(0,T;L^2(\Omega))},
\]
and then, by this estimate and \cref{eq:Rh-U}, proceeding as in the proof of
\cref{eq:conv-f-L2-Hs-L2} yields
\begin{equation*}
    \nm{R_hu-U}_{C([0,T];\dot H^1(\Omega))}
    \lesssim
    \big(\tau^{(\alpha-1)/2}+h^{1-1/\alpha} \big)
    \nm{f}_{L^{2}(0,T;L^2(\Omega))}.
\end{equation*}
Therefore, since \cref{thm:regu-pde} gives
\[
\nm{u-R_hu}_{C([0,T];\dot H^1(\Omega))}
\lesssim h^{1-1/\alpha}
\nm{f}_{L^{2}(0,T;L^2(\Omega))},
\]
we obtain
\begin{align}
    \nm{u-U}_{C([0,T];\dot H^1(\Omega))}
    & \leqslant
    \nm{u-R_hu}_{C([0,T];\dot H^1(\Omega))}+
    \nm{R_hu-U}_{C([0,T];\dot H^1(\Omega))} \notag \\
    & \lesssim
    \big(\tau^{(\alpha-1)/2}+h^{1-1/\alpha} \big)
    \nm{f}_{L^{2}(0,T;L^2(\Omega))}. \label{eq:conv-H1-inf}
\end{align}

Finally, combining \cref{eq:conv-f-L2-Hs-L2,eq:conv-H1-inf} proves
\cref{eq:conv-f-L2} and thus concludes the proof of this theorem.
\hfill\ensuremath{\blacksquare}
\section{Numerical Experiments}
\label{sec:numer}
In this section, we present some numerical examples to validate
\cref{thm:conv-f-Hs,thm:conv-f-L2} in one dimensional case. We set $ \Omega := (0,1)
$, $T:=1 $, and use the uniform temporal and spatial grids. Define
\begin{align*}
  \mathcal E_1(U):={}&
\nm{\tilde u-U}_{C([0,T];\dot H^1(\Omega))},\\
  \mathcal E_2(U):={}&{}\,\big\|
    \D_{0+}^{(\alpha-1)/2}(\tilde u-U)'
  \big\|_{L^2(0,T;L^2(\Omega))},
\end{align*}
where $ \tilde u $ is the numerical solution with $ \tau=2^{-17} $ and $ h = 2^{-10}
$. Here we observe that \cref{lem:regu-frac-deriv} implies
\[
  \nm{(\tilde u-U)'}_{{}_0H^{(\alpha-1)/2}(0,T;L^2(\Omega))}
  \sim \big\|
    \D_{0+}^{(\alpha-1)/2} \left(\tilde u-U\right)'
  \big\|_{L^{2}(0,T;L^2(\Omega))}.
\]
It is easy to see that \cref{eq:algor_sol} yields a block triangular Toeplitz-like
with tri-diagonal block system, so that we can apply a fast direct $\mathcal
O(h^{-1}J(\log J)^2)$ solver based on the divide-and-conquer strategy \cite{Ke2015}
to solve this system efficiently.
Additionally, the calculation of $\mathcal E_2(U)$
involves only the matrix-vector multiplication of a block triangular Toeplitz-like matrix, which can be completed within computational cost of $\mathcal O(h^{-1}J\log J)$ by fast Fourier transform.

\vskip0.2cm\noindent{\bf Example 1.}
This example adopts
\[
f(x,t) = t^{-0.49}x^{-0.49},
\quad (x,t)\in\Omega\times (0,T).
\]
The relationship between the spatial errors
and the spatial step sizes are displayed in \cref{fig:ex2-space}.
These numerical results indicate that
\[
\mathcal E_1(U)
\approx\mathcal O\big(h^{1-1/\alpha}\big),\quad
\mathcal E_2(U)
\approx\mathcal O\big(h^{1-1/\alpha}\big).
\]
The relationship between the errors and the
temporal step sizes are plotted in
\cref{fig:ex2-temp}, which demonstrate that
\[
\mathcal E_1(U)
\approx\mathcal O\big(\tau^{(\alpha-1)/2}\big),\quad
\mathcal E_2(U)
\approx\mathcal O\big(\tau^{(\alpha-1)/2}\big).
\]
Therefore, if $1<\alpha\leqslant3/2$ then numerical results coincide well with \cref{thm:conv-f-L2}. However, for $3/2<\alpha<2$,
numerical results also show the optimal accuracy of $\mathcal E_1(U)$ and $\mathcal E_2(U)$ with respect to the regularity, without the restriction that $h\leqslant C\tau^{\alpha/2}$.
\begin{figure}[H]
    \centering
    \includegraphics[width=350pt]{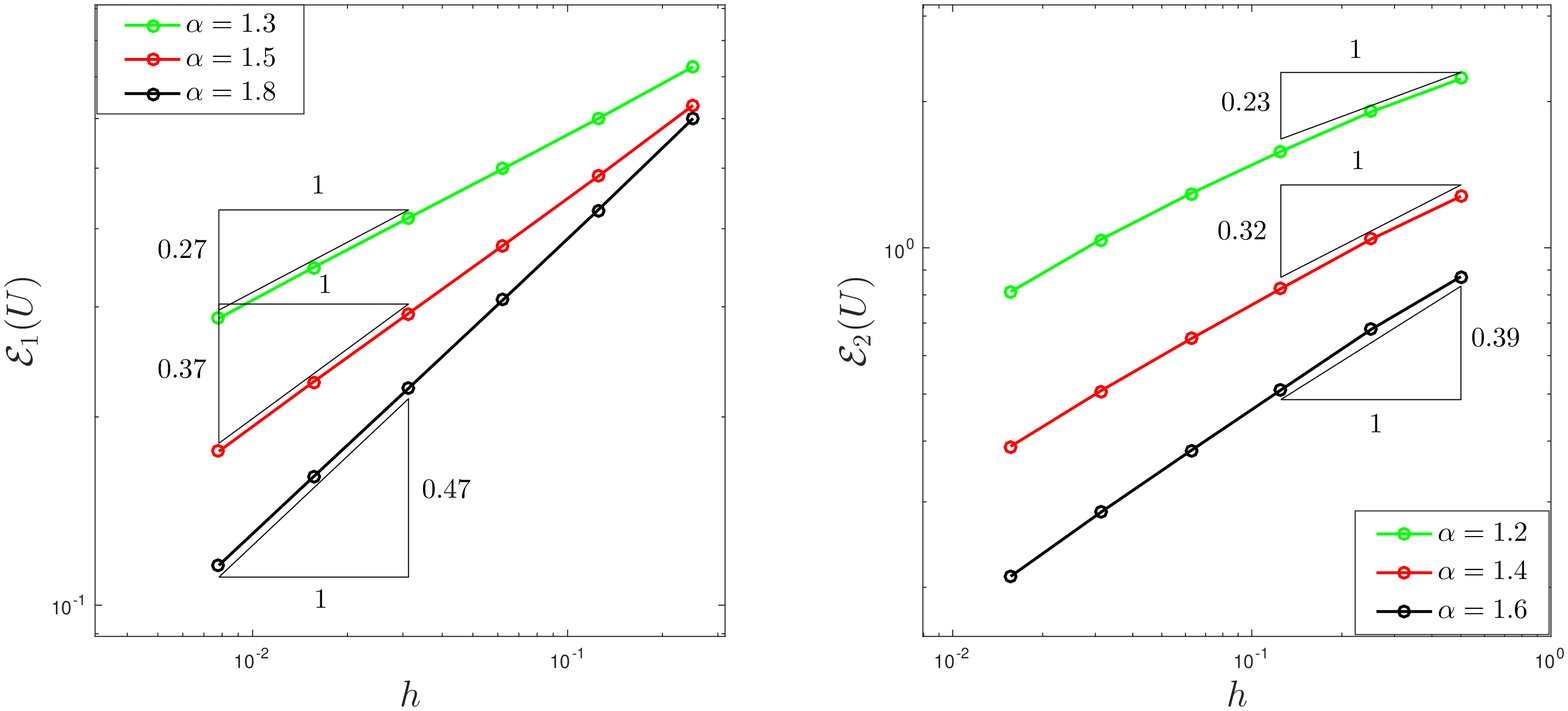}
    \caption{Spatial errors of Example 1,~$\tau=2^{-17}$.}
    \label{fig:ex2-space}
\end{figure}
\begin{figure}[H]
    \centering
    \includegraphics[width=350pt]{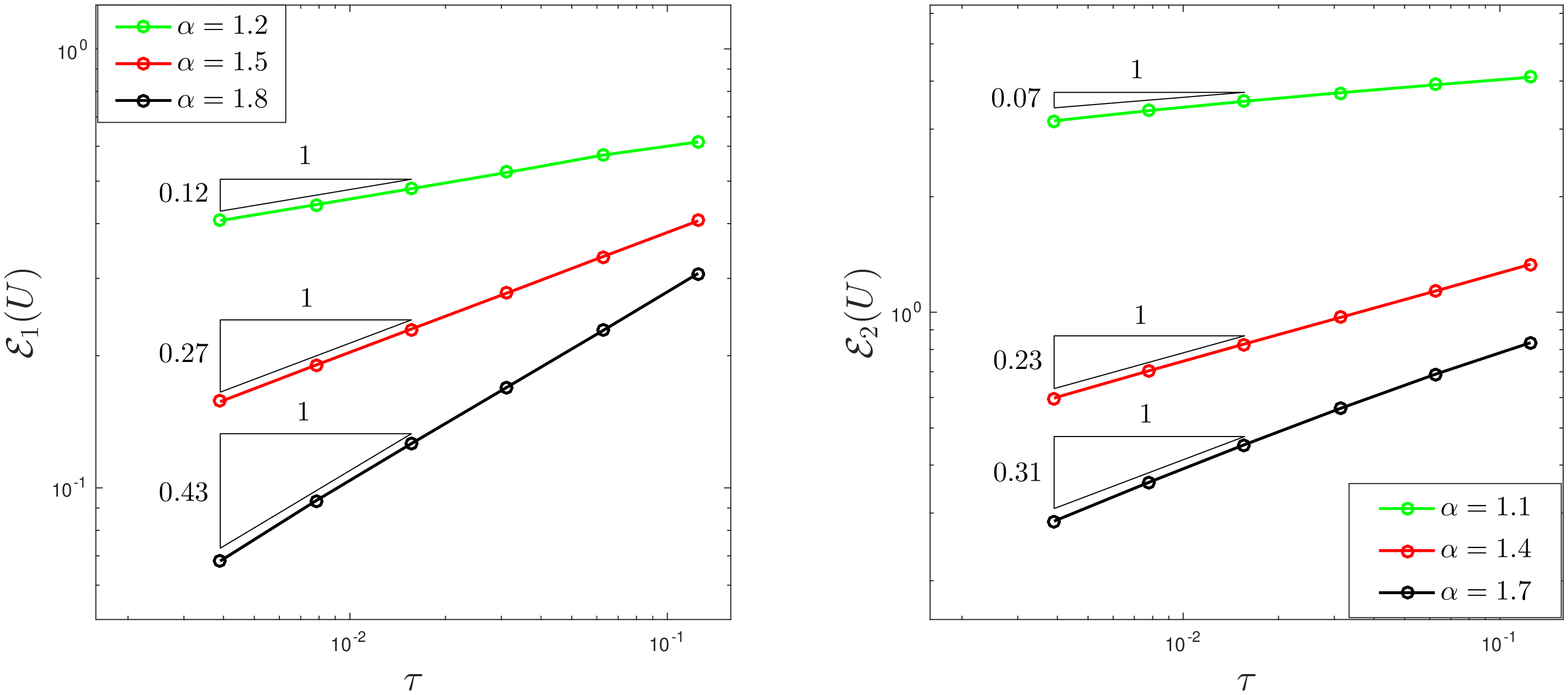}
    \caption{Temporal errors of Example 1,~$h=2^{-10}$.}
    \label{fig:ex2-temp}
\end{figure}

\vskip0.2cm\noindent{\bf Example 2.} This example employs
\[
  f(x,t) = t^{1.51-\alpha}x^{-0.49},
  \quad (x,t)\in\Omega\times (0,T),
\]
and the spatial errors and temporal errors are plotted in
\cref{fig:ex1-space,fig:ex1-temp}, respectively. These numerical results demonstrate
that
\begin{align*}
  \mathcal E_1(U)\approx {}&
\left\{
\begin{aligned}
&  \mathcal O\big(\tau^{(3-\alpha)/2}+h\big)&&{\rm~if~} 1<\alpha\leqslant 3/2,\\
&  \mathcal O\big(\tau^{(3-\alpha)/2}+h^{3/\alpha-1}\big)&&{\rm~if~} 3/2<\alpha<2,
\end{aligned}
\right.
\\
  \mathcal E_2(U)\approx {}&
\mathcal O\big(\tau^{(3-\alpha)/2}+h^{3/\alpha-1}\big).
\end{align*}
Hence, if $1<\alpha\leqslant3/2$, then numerical results verify the theoretical
predictions of \cref{thm:conv-f-Hs}. But for $3/2<\alpha<2$, numerical results also indicate that the convergence rates of $\mathcal E_1(U)$ and $\mathcal E_2(U)$ are optimal with respect to the regularity, without the requirement $h\leqslant C\tau^{\alpha/2}$.
\begin{figure}[H]
  \centering
  \includegraphics[width=350pt]{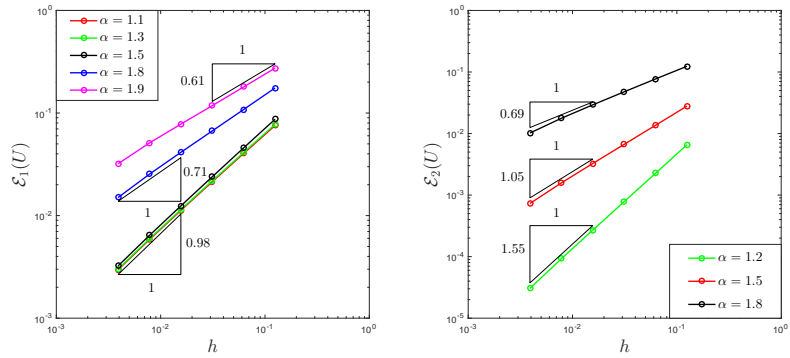}
  \caption{Spatial errors of Example 2,~$\tau=2^{-17}$.}
  \label{fig:ex1-space}
\end{figure}
\begin{figure}[H]
  \centering
  \includegraphics[width=350pt]{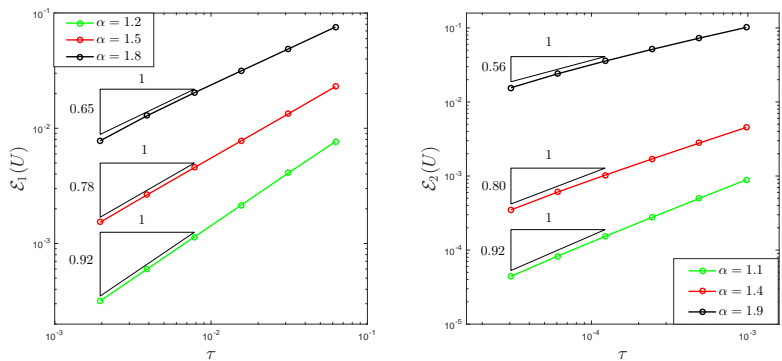}
  \caption{Temporal errors of Example 2,~$h=2^{-10}$.}
  \label{fig:ex1-temp}
\end{figure}
\section{Concluding remarks}
\label{sec:con}
This paper concerns the convergence of a Petrov-Galerkin method for fractional wave
problems with nonsmooth data. The weak solution and its regularity are studied by the
variational approach. Optimal error estimate with respect to the regularity of the
solution under the norm $ C([0,T];\dot H^1(\Omega)) $ is derived if $ f \in
L^2(0,T;L^2(\Omega)) $ and $ 1 < \alpha \leqslant 3/2 $, and numerical results
validate this theoretical result. For $ 3/2 < \alpha < 2 $, similar optimal error
estimate is also derived under the restriction that the temporal grid is quasi-uniform
and $ h \leqslant C \tau^{\alpha/2} $; however, numerical results demonstrate that the
restriction $ h \leqslant C \tau^{\alpha/2} $ is unnecessary.

In addition, optimal error estimates with respect to the regularity of the solution
or the degrees of polynomials used in the discretization under the norm $
C([0,T];L^2(\Omega)) $ have not been established, and this is our ongoing work.
\bibliographystyle{plain}

\end{document}